\documentclass[12pt, twoside, leqno]{article}

\usepackage{enumitem}

\usepackage[T1]{fontenc}
\usepackage{amsfonts,latexsym,amssymb,amsmath,amsthm,mathrsfs}

\newtheorem{thm}{Theorem}[section]
\newtheorem{lem}[thm]{Lemma}%[section]
\theoremstyle{definition}
\newtheorem{defn}[thm]{Definition}%[section]
\numberwithin{equation}{section}

\newcommand{\e}{\varepsilon}

\newcommand{\al}{\alpha}

\newcommand{\R}{\mathbb{R}}

\newcommand{\N}{\mathbb{N}}

\newcommand{\C}{\mathbb{C}}

\newcommand{\D}{\mathscr{D}}%{\mathcal{D}}

\newcommand{\bmo}{{\rm BMO}}

\newcommand{\vmo}{{\rm VMO}}
\newcommand{\cmo}{{\rm CMO}}

\newcommand{\intav}{-\!\!\!\!\!\!\int}

\DeclareMathOperator{\supp}{supp}

\def\fz{\infty}

\def\dz{\delta}

\def\wz{\widetilde}

\def\ls{\lesssim}
\def\gs{\gtrsim}

\def\lpz{{L^p(\R)}}

\def\Clm{C_{\Gamma}}

\def\r{\right}
\def\lf{\left}

\def\noz{\nonumber}

\allowdisplaybreaks

\begin{document}

%%%%% To ease editing, for IMPAN journals add:

\baselineskip=17pt

%%%%%%%%%%%%%%%%

\title{The Cauchy integral,\\ Bounded and Compact Commutators}

\author{
Ji Li\\ Department of Mathematics\\ Macquarie University\\ NSW 2109, Australia\\
E-mail: ji.li@mq.edu.au
\and
Trang T.T. Nguyen\\ School of Information Technology and Mathematical Sciences\\ University of South Australia\\ Mawson Lakes SA 5095, Australia\\
E-mail: trang.t.nguyen1@mymail.unisa.edu.au
\and
Lesley A. Ward\\ School of Information Technology and Mathematical Sciences\\ University of South Australia\\ Mawson Lakes SA 5095, Australia\\
E-mail: lesley.ward@unisa.edu.au
\and
Brett D. Wick\\ Department of Mathematics\\ Washington University -- St. Louis\\ St. Louis, MO 63130-4899 USA\\
E-mail: wick@math.wustl.edu
}

\date{}

\maketitle

%% Classification and key words; note that the 2010 classification is used:

\renewcommand{\thefootnote}{}



\renewcommand{\thefootnote}{\arabic{footnote}}
\setcounter{footnote}{0}

%%%%%%%%

\begin{abstract}
We study the commutator of the well-known Cauchy integral with
a locally integrable function $b$ on $\mathbb R$, and establish
a characterisation of the BMO space on $\mathbb R$ via the
$L^p$ boundedness of this commutator. Moreover, we also
establish a characterisation of the VMO space on~$\mathbb R$
via the  compactness of this commutator.
\end{abstract}

\section{Introduction and Statement of Main Results}
\setcounter{equation}{0} The \emph{commutator} of a singular
integral operator~$T$ with a function~$b$ is defined by
\[[b,T](f) := bT(f) - T(bf).\]
Commutators arise in various contexts. Here we focus on their
use in characterising the $\bmo$ and~$\vmo$ spaces of functions
of bounded and vanishing mean oscillation, respectively. The
first characterisation of~$\bmo$ via boundedness of commutators
is due to Coifman, Rochberg and Weiss~\cite{CRW76}. They showed
that a function~$b$ is in~$\bmo(\R)$ if and only if the
commutator~$[b,T]$ is bounded on~$L^p(\R^n)$, where~$T$ is a
convolution singular integral operator (SIO). The first
characterisation of~$\vmo$ via compactness of commutators is
due to Uchiyama~\cite{Uch78}. He showed that a function~$b$ is
in~$\vmo(\R)$ if and only if the commutator~$[b,T]$ is compact
on~$L^p(\R^n)$, where~$T$ is a convolution SIO. Since then,
many other proofs of these fundamental results have appeared,
and they have been extended to various settings. Specifically,
the commutators considered are with certain singular integral
operators, including linear, nonlinear and multilinear
operators acting on a variety of underlying spaces. See for
example \cite{Blo85, FL, KL2, HLW17, LW17, LOR, BDMT15, BT13,
CT15,TYY} and the references therein.
%Since then, many authors have focused on the boundedness and compactness of commutators with certain singular integral operators, including linear, nonlinear and multilinear operators on various underlying spaces.

The purpose of this paper is to establish such
characterisations when the operator~$T$ is the well-known
Cauchy integral~$\Clm$, which is a particular example of a
non-convolution operator. We state our main results as follows,
starting with the boundedness result.

\begin{thm}\label{thm1}
Suppose $b\in \cup_{1<q<\infty }L^q_{\textup{loc}}(\mathbb R)$
and suppose $p \in (1,\infty)$.  Then the following results
hold.

1. If $b$ is in $\bmo(\mathbb{R})$, then $[b,\Clm]$ is bounded
on $L^p(\mathbb R)$ with
  $$ \|[b,\Clm]: {L^p(\mathbb R)\to L^p(\mathbb R)}\| \leq C_1\|b\|_{\bmo(\mathbb{R})}. $$

2. If $[b,\Clm]$ is bounded on $L^p(\mathbb R)$, then $b$ is in
$\bmo(\mathbb{R})$ with
 $$ \|b\|_{\bmo(\mathbb{R})}  \leq C_2 \|[b,\Clm]: {L^p(\mathbb R)\to L^p(\mathbb R)}\|. $$
\end{thm}

We also establish the following compactness result.

\begin{thm}\label{thm2}
Suppose $b\in \bmo(\mathbb{R})$  and suppose $p \in
(1,\infty)$. Then the following results hold.

1. If $b$ is in $\vmo(\mathbb{R})$, then $[b,\Clm]$ is a
compact operator on $L^p(\mathbb R)$.

2. If $[b,\Clm]$ is a compact operator on $L^p(\mathbb R)$,
then $b$ is in $\vmo(\mathbb{R})$.
\end{thm}

In this paper, all functions considered are real-valued
functions defined on~$\R$. For every $x \in \R$, $r\in
\mathbb{R}_+$, we define the interval
$$I(x, r):=(x-r, x+r).$$
For~$\lambda > 0$ we define the dilate~$\lambda J$ of an
interval~$J$ to be the interval with the same midpoint as~$J$
and length~$\lambda|J|$. In particular, $\lambda I(x, r) = I(x,
\lambda r)$ for all~$x \in \R, r\in \mathbb{R}_+$ and~$\lambda
> 0$. Given~$y \in \R$, we define the
translate $J+y := \{x+y : x\in J\}$ of an interval~$J$. We use
the notation~$\intav_I f(x) \,dx := \frac{1}{|I|}\int_I f(x)
\,dx$. Throughout the paper, we denote by $C$ and
$\widetilde{C}$ {positive constants} that are independent of
the main parameters, but that may vary from line to line. For
$p\in(1, \fz)$, $p'$ means the conjugate of $p$: $1/p'+1/p=1$.
If $f\le Cg$, we write $f\ls g$ or $g\gs f$; and if $f \ls g\ls
f$, we  write $f\sim g.$

This paper is organised as follows. In Section~\ref{sec:defns},
we recall some definitions and theorems which will be used in
the proofs of our results. In Section~\ref{sec:boundedness}, we
prove our first result, which is about the relationship
between~$\bmo$ functions and the boundedness of the commutator.
In Section~\ref{sec:compact}, we prove our second result, which
is about the relationship between~$\vmo$ functions and the
compactness of the commutator.

\section{Preliminaries}\label{sec:defns}

In this section we recall the space $\bmo$ of functions of
bounded mean oscillation, the space $\vmo$ of functions of
vanishing mean oscillation, singular integral operators,
Calder\'on--Zygmund operators, the Frech\'et--Kolmogorov
theorem, and the Cauchy integral.

\subsection{$\bmo$ and $\vmo$ spaces}
\begin{defn}\textup{(}\textbf{$\bmo$}\textup{)}\label{def:BMO}
    A locally integrable real-valued function $f: \mathbb{R}
    \rightarrow \mathbb{R}$ is said to be of \emph{bounded mean
    oscillation}, written $f \in \bmo$ or $f \in \bmo(\mathbb{R})$,
    if
    \[
        \|f\|_{\bmo}
        := \sup_{\substack{x\in \R,\\r>0}} M(f,I(x,r))
        := \sup_I\frac{1}{|I|}\int_I \lf|f(x) - f_I\r| \,dx < \infty,
    \]
    where
    \[
        f_I
        := \frac{1}{|I|}\int_I f(y) \,dy
    \]
    is the average of the function~$f$ over the interval~$I$. Here
    $I$ denotes an interval in $\mathbb{R}$.
\end{defn}

We denote by~$\vmo(\R)$ the space of functions of
\emph{vanishing mean oscillation}, defined to be
the~$\bmo(\mathbb{R})$-closure of the set~$ \D := C_c^{\infty}(\R)$
of~$C^{\infty}(\R)$ functions with compact support. We note
that our definition of~$\vmo(\R)$ is the same as that
of~$\vmo(\R)$ in~\cite{CW77}, as well as that of~$\cmo(\R)$
in~\cite{Uch78}.

There are several characterisations of~$\vmo$ in the
literature. Here we use the characterisation appearing
in~\cite{Daf02}.

\begin{defn}\textup{(}\textbf{$\vmo$}\textup{)}\label{def:VMODafni} \cite{Daf02}
A $\bmo$ function~$f:\mathbb{R}\rightarrow\mathbb{R}$ is said to
be of \emph{vanishing mean oscillation}, written~$f \in \vmo$
or~$f \in \vmo(\mathbb{R})$, if

$\textup{(1)}
  \quad \displaystyle\lim_{\delta\rightarrow0}\sup_{I,|I|\leq \delta}\frac{1}{|I|}\int_I|f(x)-f_I|\,dx=0, $\\

$\textup{(2)}
  \quad \displaystyle\lim_{R\rightarrow\infty}\sup_{I,|I| \geq R}\frac{1}{|I|}\int_I|f(x)-f_I|\,dx=0, \text{ and}$\\

$\textup{(3)}
  \quad \displaystyle\lim_{R\rightarrow\infty}\sup_{I,I\cap I(0,R)=\emptyset}\frac{1}{|I|}\int_I|f(x)-f_I|\,dx=0.$
\end{defn}

In~\cite[Lemma, Section~3, p.166]{Uch78}, Uchiyama
characterises the space $\vmo(\mathbb{R}^n)$ (denoted there by
$\cmo(\mathbb{R}^n)$) in terms of three conditions similar to
those in Definition~\ref{def:VMODafni}, but which are expressed
in terms of the quantity $\inf_c \frac{1}{|Q|} \int_Q |f(y) -
c| \, dy$. As we note in Section~\ref{sec:compact}, this
infimum is attained when $c$ is any median~$\alpha_I(f)$ of $f$
on~$I$. It is straightforward to show directly that Uchiyama's
definition is equivalent to Definition~\ref{def:VMODafni}.

We refer the reader to Bourdaud's paper~\cite{Bou02} for a
careful treatment of various $\bmo$ and $\vmo$ spaces, and in
particular a clarification of the confusion of the $\vmo$ and
$\cmo$ notation.

\subsection{Singular Integral Operators}

\begin{defn} \label{CZR} \cite{Chr55}
\textup{(}\textbf{Standard kernel}\textup{)} A \emph{kernel $K$
on $\R$} is a function $K: \R \times \R \rightarrow \R$. A
kernel $K$ is said to \emph{satisfy standard estimates} if
there exist $\delta > 0$ and $C < \infty$ such that for all
distinct $x,y \in \R$ and all $y'$ with $|y-y'| < |x-y|/2$ we
have:

{\rm (i)} $|K(x,y)| \leq C|x-y|^{-1}$,

{\rm (ii)} $|K(x,y) - K(x,y')| \leq C\big( \frac{|y-y'|}{|x-y|}\big)^{\delta} |x-y|^{-1}$, \text{ and}

{\rm (iii)} $|K(y,x) - K(y',x)| \leq C\big( \frac{|y-y'|}{|x-y|}\big)^{\delta} |x-y|^{-1}$.

The smallest constant $C$ for which properties \textup{(i)--(iii)} hold is denoted by $|K|_{CZ}$.
\end{defn}

\begin{defn} \cite{Chr55} (\textbf{Operators associated to a kernel})
Let $\D'$ denote the space of distributions dual to $\D =
C_c^{\infty}(\R)$. A continuous linear operator $T:
C_c^{\infty}(\R) \rightarrow \D'$ is said to be
\emph{associated to a kernel $K$} if whenever $f, g \in
C_c^{\infty}(\R)$ have disjoint supports,  we have
\[\langle Tf,g\rangle = \int_{\R}  \int_{\R} K(x,y)f(y)g(x) \,dy dx.\]
Here the brackets denote the natural pairing of $\D'$ with
$C_c^{\infty}(\R)$.   Since $Tf$ is in the dual $\D'$ of
$C_c^{\infty}(\R)$, it is a bounded linear functional that acts
on functions $g$ in $C_c^{\infty}(\R)$.
\end{defn}

\begin{defn} \cite{Chr55} \label{SIOinR} (\textbf{Singular integral operators on $\R$})
    A \emph{singular integral operator (SIO)}\emph{ on~$\R$} is a
    continuous linear mapping from $C_c^{\infty}(\R)$ to $\D'$
    which is associated to a standard kernel.
\end{defn}

\begin{defn} \cite{Chr55} (\textbf{Calder\'{o}n--Zygmund operators on $\R$})
    Let $T$ be a SIO on $\R$. $T$ is  a \emph{Calder\'{o}n--Zygmund
    operator (CZO)} \emph{on $\R$} if it extends to a bounded
    operator from $L^2(\R)$ to itself.
\end{defn}

A SIO $T$ is \emph{bounded} from $L^p(\R)$ to $L^p(\R)$, for~$p
\in (0,\infty)$, if there exists a constant $C$ such that
$\|Tf\|_p \leq C \|f\|_p$ for all $f\in L^p(\R)$. A SIO $T$ is
\emph{compact} on $L^p(\R)$ if for all bounded sets~$E \subset
L^p(\R)$, $T(E)$ is precompact. A set $S$ is \emph{precompact}
if its closure is compact. A common way to check precompactness
is to use the criteria established in the well known
Frech\'{e}t--Kolmogorov theorem
\cite[p.275]{Yos80},~\cite[pp.111--114]{Bre10}.

\begin{thm}\label{t-fre kol}(\textbf{Frech\'{e}t--Kolmogorov theorem})
For  $1<p<\infty$, a subset $E$ of $L^p(\mathbb R)$ is totally bounded (or precompact) if and only if
the following three statements hold:

{\rm(a)}\ $E$ is uniformly bounded, i.e., $\sup\limits_{f\in
E}\|f\|_{L^p(\mathbb R)}<\infty$;

{\rm(b)}\ $E$ vanishes uniformly at infinity, i.e., for every
$\e>0$, there exists a compact region~$K_{\e}$
 such that for every $f\in E$,
$\|f\|_{L^p(K^c_\e)} < \e;$ and

{\rm(c)}\  $E$ is uniformly equicontinuous, i.e., for every
$f\in E$, $\lim\limits_{|z| \to 0} \|f(\cdot+z) -
f(\cdot)\|_{\lpz} = 0.$
\end{thm}

\subsection{Cauchy Integral \vspace{-0.5cm}} \hfill %\break

Suppose $\Gamma$ is a curve in the complex plane~$\C$ and~$f$
is a function defined on the curve~$\Gamma$. The \emph{Cauchy
integral} of~$f$ is the operator $\mathcal{C}_{\Gamma}$ given
by
\begin{equation}\label{eq.11.0}
      \mathcal{C}_{\Gamma}(f)(z)
    := \frac{1}{2\pi i} \int_{\Gamma}\frac{f(\zeta)}{\zeta - z}\,d\zeta.
\end{equation}
A curve~$\Gamma$ is said to be a \emph{Lipschitz curve} if it
can be written in the form~$\Gamma = \{x + iA(x): x \in \R\}$
where~$A: \R \rightarrow \R$ satisfies a Lipschitz condition
\begin{equation}\label{eq:l1}
  |A(x_1) - A(x_2)| \leq L|x_1 - x_2| \quad \text{for all } x_1, x_2 \in \R.
\end{equation}
The best constant~$L$ in~\eqref{eq:l1} is referred to as the
\emph{Lipschitz constant} of~$\Gamma$ or of~$A(x)$. One can
show that $A$~satisfies a Lipschitz condition if and only
if~$A$ is differentiable almost everywhere on~$\R$ and~$A'\in
L^\infty(\mathbb R)$. The Lipschitz constant is~$L =
\|A'\|_{\infty}$.

The \emph{Cauchy integral associated with the Lipschitz
curve}~$\Gamma$ is the SIO~$\widetilde{C}_{\Gamma}$ given by
\begin{equation}\label{eq:l2}
  \widetilde{C}_{\Gamma}(f)(x)
  := {\rm p.v.} \frac{1}{\pi i}\int_{\R} \frac{(1 + iA'(y))f(y)}{y-x + i(A(y) - A(x))}\,dy,
\end{equation}
where $f \in C_{c}^{\infty}(\R)$. The kernel
of~$\widetilde{C}_{\Gamma}$ is
\[\widetilde{C}_{\Gamma}(x,y) =  \frac{1}{\pi i}\frac{1 + iA'(y)}{y-x + i(A(y) - A(x))},\]
which is not a standard kernel because the function~$1 + iA'$
does not necessarily possess any smoothness. As noted
in~\cite[p.289]{Gra04}, the~$L^p$-boundedness
of~$\widetilde{C}_{\Gamma}$ is equivalent to that of the
related operator $C_{\Gamma}$ defined by
\begin{equation}\label{eq:l3}
  C_{\Gamma}(f)(x)
  := {\rm p.v.} \frac{1}{\pi i}\int_{\R} \frac{f(y)}{y-x + i(A(y) - A(x))}\,dy.
\end{equation}
Moreover, as we will see in Lemma~\ref{lem:Cbded}, the kernel
\begin{equation}\label{eq:Ckernel}
    \Clm(x,y)
    = \frac{1}{\pi i}\frac{1}{y-x + i(A(y) - A(x))}
\end{equation}
of~$ C_{\Gamma}$ satisfies standard estimates, and
while~$C_{\Gamma}(f)$ is initially defined for $f \in
C_{c}^{\infty}(\R)$, the operator~$C_{\Gamma}$ can be extended
to all~$f \in L^{p}(\R)$, for each $p \in (1,\infty)$.

Note that the operators $\mathcal{C}_{\Gamma}$,
$\widetilde{C}_{\Gamma}$ and $C_{\Gamma}$ defined
in~\eqref{eq.11.0}, \eqref{eq:l2} and~\eqref{eq:l3} are all
distinct. For the rest of this paper, we work with the
operator~$\Clm$ given by equation~\eqref{eq:l3}; we call~$\Clm$
the \emph{Cauchy integral}. Also, for convenience we omit the
factor~$1/(\pi i)$ from its kernel from here on.

\section{Proof of Theorem \ref{thm1}: Boundedness of~$[b,\Clm]$}
\label{sec:boundedness} In this section, we prove our first
result, which is about the boundedness of the
commutator~$[b,\Clm]$. The main ingredient in the proof of
Theorem~\ref{thm1} is the characterisation of the function
space~$\bmo(\R^n)$  via commutators in a multilinear
($m$-linear) setting. The necessity of the~$\bmo$ condition was
proved in~\cite{CRW76} in the linear setting~($m=1$) on~$\R^n,
n \geq 1$. For the $m$-linear setting on~$\R^n, n \geq 1$, it
was proved in \cite{Cha16}. The sufficiency of the~$\bmo$
condition in the $m$-linear setting on~$\R^n, n \geq 1$ was
shown in~\cite{LOPTT09, PT03, Tan08}. These results are also
stated as  Theorem~1.4 in~~\cite{LW17}. See also the recent
paper \cite{LOR}. In this paper, we work in the linear setting
($m=1$) with the real line~($n=1$) being the underlying space.
Below we state these results in the special case where~$m=n=1$.

\begin{thm}\label{c:bmo}~\cite{CRW76}
    Suppose that $T$ is an $L^p$-bounded SIO for some~$p$ with~$1<p<\infty$.
    %\label{c:bmo}
    If $b$ is in $\rm BMO(\mathbb{R})$, then  the commutator
    $[b,T]$ is a bounded map from $L^{p}(\mathbb{R})$ to
    $L^{p}(\mathbb{R})$ for all~$p$ with~$1<p<\infty$, with
    $$
    \|[b,T]: L^{p}(\mathbb{R})\to L^{p}(\mathbb{R})\| \leq C\|b\|_{\rm BMO(\mathbb{R})}.
    $$
\end{thm}

\begin{thm}\label{c:bmo1}~\cite{LOPTT09, PT03, Tan08}
    Suppose that $b\in L^p_{\textup{loc}}(\mathbb{R})$ and $T$ is
    1-1-homogeneous. If  $[b,T]$ is bounded from
    $L^{p}(\mathbb{R})$ to $L^{p}(\mathbb{R})$ for some~$p$ with
    $1< p<\infty$, then $b$ is in $\rm BMO(\mathbb{R})$ with $$
    \|b\|_{\rm BMO(\mathbb{R})} \leq C \| [b,T]:
    L^{p}(\mathbb{R})\to L^{p}(\mathbb{R})\|.$$
\end{thm}

\begin{defn}
A SIO $T$  is called \emph{$m$-$n$-homogeneous} if there exists
a constant~$C>0$ such that for all~$M>10$, for all~$r>0$, and
for all collections of~$m+1$ pairwise disjoint balls
$B_0(x_0,r), \ldots, B_m(x_m,r)$ in~$\R^n$ satisfying the condition
\[
    |y_0 - y_l|
    \sim Mr
    \quad \text{for all } y_0 \in B_0
    \text{ and for all } y_l \in B_l, l = 1,2,\ldots,m,
\]
we have
\[
    |T(\chi_{B_1},\ldots,\chi_{B_m})(x)|
    \geq \frac{C}{M^{mn}} \quad \text{for all } x \in B_0(x_0,r).
\]
\end{defn}

We note again that in this paper, ~$m=n=1$. If we can show that
the Cauchy integral~$C_{\Gamma}$ satisfies the hypotheses of
Theorems~\ref{c:bmo} and~\ref{c:bmo1}, then Theorem~\ref{thm1}
is proved. Specifically, if we can show that~$C_{\Gamma}$ is an
$L^p$-bounded SIO for some~$p \in (1, \infty)$, then the first
part of Theorem~\ref{thm1} is proved. Similarly, if we can show
that~$C_{\Gamma}$ is 1-1-homogeneous, then the second part of
Theorem~\ref{thm1} is proved. These results are presented in
Lemma~\ref{lem:Cbded} and~\ref{lem:Chomo}.

\begin{lem}\label{lem:Cbded}
The Cauchy integral~$C_{\Gamma}$ is an $L^p$-bounded SIO, for
every $p \in (1, \infty)$.
\end{lem}

\begin{proof}
The Cauchy integral~$C_{\Gamma}$ is a SIO if it is a continuous
linear mapping from $C_c^{\infty}(\R)$ to $\D'$ that is
associated to a standard kernel. Recall that the kernel
of~$C_{\Gamma}$ is
\[
    \Clm(x,y)
    = \frac{1}{y-x + i(A(y) - A(x))}.
\]
In Example 4.1.6 in~\cite{Gra04}, it is noted that~$\Clm(x,y)$
is a standard kernel. In particular,~$\Clm(x,y)$ has the
following properties, for all~$x, y, y' \in \R$ such that
$|y-y'| \leq \frac{1}{2}|y-x|$:
\begin{equation}\label{eq:thm1.1}
  |\Clm(x,y)| \leq \frac{1}{|y-x|},
\end{equation}

\begin{equation}\label{eq:thm1.2}
  |\Clm(x,y) - \Clm(x,y')| \leq \frac{2(L+1)|y'-y|}{|y-x |^2}
\end{equation}
where $L>0$ is the Lipschitz constant of~$A(x)$, and
\begin{equation}\label{eq:thm1.3}
  |\Clm(y,x) - \Clm(y',x)| \leq \frac{2(L+1)|y'-y|}{|y-x |^2}.
\end{equation}
Therefore, the Cauchy integral~$C_{\Gamma}$ is a SIO.

Coifman, McIntosh and Meyer~\cite{CMM82} showed that
$C_{\Gamma}$ is bounded on~$L^2$. Additionally, Calder\'{o}n
and Zygmund showed that a SIO which is bounded on~$L^2$ is also
bounded on~$L^p$ for all~$p \in (1,\infty)$. Thus the Cauchy
integral~$C_{\Gamma}$ is bounded on ~$L^p$ for all~$p \in
(1,\infty)$. This completes the proof of Lemma~\ref{lem:Cbded}.
\end{proof}

\begin{lem}\label{lem:Chomo}
The Cauchy integral~$C_{\Gamma}$ is 1-1-homogeneous.
\end{lem}

\begin{proof}
We need to show that there exists a constant~$C>0$ such that
for all~$M>10$, for all~$r>0$, and for all  disjoint
intervals~$I_0=I_0(x_0,r)$ and~$I_1=I_1(x_1,r)$ satisfying the
condition
\begin{equation}\label{eq:thm1.4}
|y_0 - y_1| \sim Mr \quad \text{for all } y_0 \in I_0, y_1 \in I_1,
\end{equation}
we have
\[|C_{\Gamma}(\chi_{I_1})(x)| \geq \frac{C}{M} \quad \text{for all } x \in I_0(x_0,r).\]

Fix an~$M>10$, $r>0$, and disjoint intervals~$I_0=I_0(x_0,r)$
and~$I_1=I_1(x_1,r)$ satisfying the
condition~\eqref{eq:thm1.4}. Note that by the choice of the
intervals~$I_0$ and~$I_1$, for each fixed~$y_0 \in I_0$  we
have either
%for all~$y \in I_1$,
\[y_1 > y_0 \quad \text{for all } y_1 \in I_1, \quad \text{ or }
\quad y_1 < y_0 \quad \text{for all } y_1 \in I_1.\] We will
consider the case~$y_1 > y_0$. The case~$y_1<y_0$ follows
exactly the same reasoning. Now for each~$x \in I_0$, $y \in
I_1$ satisfying the condition $|x-y| \sim Mr$  and Lipschitz
function~$A$ we have
\begin{align}\label{eq:lem2.4}
  |C_{\Gamma}(\chi_{I_1})(x)| &= \bigg|{\rm p.v.} \int_{\R} \frac{\chi_{I_1}(y)}{y-x + i(A(y) - A(x))}\,dy\bigg|\noz\\
   &= \bigg|\int_{y \in I_1} \frac{y-x  -i(A(y) - A(x))}{(y-x)^2 + (A(y) - A(x))^2} \,dy\bigg|\noz\\
   &\geq  \int_{y \in I_1} \frac{y-x }{(y-x)^2 + (A(y) - A(x))^2} \,dy \noz \\
   &\gtrsim  \frac{1}{(L^2+1)} \frac{1}{Mr}|I_1|\noz\\
  % &=&  \frac{1}{\pi(L^2+1)} \frac{1}{Mr}2r \\
   &=\frac{2}{(L^2+1)} \frac{1}{M}.  \noz  %= \frac{C}{M}, \noz
\end{align}
This estimate holds for all~$M>10$, for all~$r>0$, for all
disjoint intervals~$I_0=I_0(x_0,r)$ and~$I_1=I_1(x_1,r)$
satisfying the condition~\eqref{eq:thm1.4}, so~$C_{\Gamma}$ is
1-1-homogeneous, as required.
\end{proof}

As noted above, the results of Theorems~\ref{c:bmo}
and~\ref{c:bmo1}, coupled with Lemmas~\ref{lem:Cbded}
and~\ref{lem:Chomo}, establish Theorem~\ref{thm1}.
\hfill\(\Box\)

\section{Proof of Theorem \ref{thm2}: Compactness of~$[b,\Clm]$}
\label{sec:compact} The idea of the proof of Theorem~\ref{thm2}
is originally due to Uchiyama~\cite{Uch78}. The main
ingredients of the proof are the~$\vmo$
characterisation~(Definition~\ref{def:VMODafni}) and the
Frech\'{e}t-Kolmogorov theorem (Theorem~\ref{t-fre kol}). To
prove the sufficiency in Theorem~\ref{thm2}, that is  if~$[b,
\Clm]$ is a compact operator on~$L^p(\R)$, then~$b \in
\vmo(\R)$, we use contradiction argument via
Definition~\ref{def:VMODafni}. Specifically, we show that
if~$b$ fails to satisfy any one of the conditions~(1)--(3) in
Definition~\ref{def:VMODafni}, then the commutator~$[b,\Clm]$
is not compact. To prove the necessity in Theorem~\ref{thm2},
that is  if~$b \in \vmo(\R)$, then~$[b,\Clm]$ is a compact
operator on~$L^p(\R)$, we first reduce to showing
that~$[b,\Clm]$ is compact for~$b \in C_c^{\infty}(\R)$. Then
we show that for all bounded subsets~$E \subset L^p(\R)$,
$[b,\Clm]E$ is precompact, using Theorem~\ref{t-fre kol}. This
implies that~$[b,\Clm]$ is compact on~$L^p(\R)$.

The proof of Theorem~\ref{thm2} requires lower and upper bounds
for integrals of $\lf|[b, \Clm]f_j\r|^p$ over certain intervals, where
$\{f_j\}_j$ is a certain bounded subset of $L^p(\R)$ and
$b\in\bmo(\R)$. These bounds will be obtained in
Lemma~\ref{lem:vmo-contra} below.

\begin{lem}
\label{lem:vmo-contra}
    Assume that $b\in\bmo(\R)$ with
    $\|b\|_{\bmo(\R)}=1$ and there exist $\e >0$ and a
    sequence $\{I_j\}_{j=1}^\fz:=\{I(x_j, r_j)\}_j$ of intervals
    such that for each $j \in \N$,
    \begin{equation}\label{lower bdd osci}
    M(b, I_j)= \intav_{I_j}|b(y) - b_{I_j}| \,dy>\e.
    \end{equation}
    For $j\in\N$, $k\in\N \cup \{0\}$, let
    \[
        I_j^k
        := \lf(x_j+2^kr_j,\,x_j+2^{k+1}r_j\r).
    \]
    Fix $p\in(1,\infty)$. Then there exist functions $\{f_j\}_j\subset L^p(\R)$ and
    positive constants $A_1>4$, $\wz C_1$ and $\wz C_2$
    such that for all $j \in \N$ and $k\ge\lfloor\log_2 A_1\rfloor$, we have
    \begin{equation}\label{lower upper lpbdd C comm0}
      \|f_j\|_\lpz \ls 1 \qquad \text{with constant independent of } j,
    \end{equation}
    \begin{equation}\label{lower upper lpbdd C comm}
        \int_{I_j^k}\lf|\lf[b, \Clm\r]f_j(y)\r|^p\,dy
        \geq\wz C_1\e^p \frac{1}{2^{k(p-1)}},
        \qquad \text{ and } %\frac{|I_j|^{p-1}}{|2^kI_j|^{p-1}},
    \end{equation}
    %\begin{equation}\label{lower upper lpbdd riesz comm}
    %\wz C_1\e^p\frac{\lf|I_j\r|^{p-1}}{[m_\lz(2^kI_j)]^{p-1}}\le
    %\int_{I_j^k}\lf|\lf[b, \Clm\r]f_j(y)\r|^p\ytz\le \wz C_2 \frac{\lf|I_j\r|^{p-1}}{[m_\lz(2^kI_j)]^{p-1}},
    %\end{equation}
    \begin{equation}\label{lower upper lpbdd C comm2}
        \int_{2^{k+1}I_j\setminus 2^k I_j}\lf|\lf[b, \Clm\r]f_j(y)\r|^p \,dy
        \leq \wz C_2 \frac{1}{2^{k(p-1)}}.
    %\frac{|I_j|^{p-1}}{|2^kI_j|^{p-1}}.
    \end{equation}
    The functions $\{f_j\}$ and the constants $\wz C_1$ and $\wz
    C_2$ depend on $p$ but not on $j$ or $k$, while $A_1$ is
    independent of $p$, $j$, and $k$.
\end{lem}

Before proving Lemma~\ref{lem:vmo-contra}, we recall some
results related to the \emph{median value}~$\al_I(f)$ of a
function~$f$ on an interval~$I$. See \cite{CSS12},
\cite{Ler11}, \cite[pp.160--166]{Gra04} and~\cite[p.30]{Jou83}
for more details. When~$f \in L^1_{\text{loc}}(\R)$ and~$I$ is
any interval on~$\R$, the constants~$c=\al_I(f)$ for
which~$\inf_c \frac{1}{|I|}\int_I |f(x) - c| \,dx$ is attained
are the ones that satisfy
\begin{equation}\label{eq:thm2.1}
  |\{x \in I: f(x) > \al_I(f)\}| \leq \frac{1}{2}|I| \quad \text{ and}
\end{equation}
\begin{equation}\label{eq:thm2.2}
  |\{x \in I: f(x) < \al_I(f)\}| \leq \frac{1}{2}|I|.
\end{equation}
Note that given a function~$f$ and an interval~$I$, the median $\al_I(f)$ may not be uniquely determined. In each such case, we mean by~$\al_I(f)$ a particular fixed value of the median.
Additionally, using the John--Nirenberg inequality and H\"{o}lder's inequality, we obtain that for all~$p$ with~$1\leq p<\infty$ and for all~$f \in L^1_{\text{loc}}(\R)$,
\begin{equation}\label{eq:intro1}
  \sup_I\lf(\frac{1}{|I|}\int_I \lf|f(x) - f_I\r|^p \,dx\r)^{1/p} \sim \|f\|_{\bmo}, \quad \text{and}
\end{equation}
\begin{equation}\label{eq:intro2}
  \sup_I\lf(\frac{1}{|I|}\int_I \lf|f(x) - \al_I(f)\r|^p \,dx\r)^{1/p} \sim \|f\|_{\bmo}.
\end{equation}
Also, as shown in~\cite[Equation (2.2)]{CSS12}, for each interval~$I \subset \mathbb{R}$,
%and the fact that the oscillation~$\frac{1}{|I|}\int_I  |f(x) - c| \,dx$ is minimised when the constant~$c$ is the median,
\begin{equation}\label{eq:thm2.3}
  \frac{1}{|I|}\int_I \lf|f(x) - f_I\r| \,dx  \sim  \frac{1}{|I|}\int_I \lf|f(x) - \al_I(f)\r| \,dx.
\end{equation}

Now we will prove Lemma~\ref{lem:vmo-contra}.

\noindent\emph{Proof of Lemma~\ref{lem:vmo-contra}.}
  For each $j \in \N$, define
  \begin{eqnarray*}
    f_j &:=& |I_j|^{-1/p}\lf(f^1_j-f^2_j\r), \text{ where} \\
    f^1_j &:=& \chi_{I_{j,\,1}}-\chi_{I_{j,\,2}}:=\chi_{\{x\in I_j:\, b(x)>\al_{I_j}(b)\}}-\chi_{\{x\in I_j:\, b(x)<\al_{I_j}(b)\}}, \\
    f^2_j &:=& a_j\chi_{I_j},
  \end{eqnarray*}
  and $a_j$ is a constant chosen so that
  \begin{equation}\label{fj proper-2}
 \int_{\R} f_j(x)\,dx=0.
  \end{equation}

We claim that the following properties hold:
 \begin{equation}\label{eq:thm2.10}
 |a_j|\le 1/2,
\end{equation}
\begin{equation}\label{eq:thm2.4}
\supp(f_j)\subset I_j,
\end{equation}
\begin{equation}\label{fj proper-1}
 f_j(y)\lf[b(y)-\al_{I_j}(b)\r]\ge0 \quad \text{ for all } y\in I_j, \text{ and }
 \end{equation}
 \begin{equation}\label{fj proper-3}
|f_j(y)|\sim \lf|I_j\r|^{-1/p} \quad \text{ for all } y\in (I_{j,\,1}\cup I_{j,\,2}).
\end{equation}
To see~\eqref{eq:thm2.10}, we start with equation~\eqref{fj
proper-2}. By the definition of~$f^1_j$ and~$f^2_j$, and using
property~\eqref{eq:thm2.2}  of the median, we see that
 \begin{align*}\label{eq:thm2.8}
   0 = \int_{\R} f_j(x)\,dx %&=& \int_{\R} |I_j|^{-1/p}\lf(f^1_j(x)-f^2_j(x)\r)\,dx \noz\\
    &= \int_{\R} |I_j|^{-1/p}\lf(\chi_{I_{j,\,1}}(x)-\chi_{I_{j,\,2}}(x)-a_j\chi_{I_j}(x)\r)\,dx \noz\\
    %&=& |I_j|^{-1/p} \lf(\int_{\R}{I_{j,\,1}}\,dx - \int_{\R}{I_{j,\,2}}\,dx - a_j\int_{\R}{I_{j}}\,dx\r) \noz\\
    &=  |I_j|^{-1/p} \lf(\lf|I_{j,\,1}\r| - \lf|I_{j,\,2}\r| - a_j\lf|I_{j}\r|\r) \\
&\geq |I_j|^{-1/p} \lf(\lf|I_{j,\,1}\r| - \frac{\lf|I_{j}\r|}{2} - a_j\lf|I_{j}\r|\r) \\
   %&=& |I_j|^{-1/p} \lf(\lf|I_{j,\,1}\r| - \lf(\frac{1}{2}+a_j\r)\lf|I_{j}\r|\r) \\
    &\geq-\lf(\frac{1}{2}+a_j\r)|I_j|^{-1/p}\lf|I_{j}\r| = -\lf(\frac{1}{2}+a_j\r)|I_j|^{1/p'}. %\\
    %&\geq& -\lf(\frac{1}{2}-\frac{1}{2}\r)|I_j|^{1/p'} =0.
 \end{align*}
Hence~$a_j\geq -1/2$. Similarly, using~\eqref{fj proper-2} and
property~\eqref{eq:thm2.1}  of the median, we see that~$a_j
\leq 1/2$. Hence~$|a_j|\leq 1/2$, as required.
Equation~\eqref{eq:thm2.4} is immediate from the definition
of~$f_j$.

To see~\eqref{fj proper-1}, we consider the
three~cases when~$y \in I_{j,1}$, $y \in I_{j,\,2}$, and $y \in
I_j\setminus (I_{j,\,1} \cup I_{j,\,2})$. If~$y \in I_{j,1}$,
then  by the definitions of~$f_j^1$ and~$f_j^2$ and
equation~\eqref{eq:thm2.10} we have
 \[b(y) > \al_{I_j}(b), \quad f^1_j(y) = 1 > 0\quad \text{and}\quad  f^2_j \leq \frac{1}{2}. \]
These yield
 \begin{equation}\label{eq:thm2.14}
    f_j(y)\lf[b(y)-\al_{I_j}(b)\r]>0\quad \text{for all } y \in I_{j,\,1}.
 \end{equation}
The case of~$y \in I_{j,\,2}$ is similar.
Next, if~$y \in I_j\setminus (I_{j,\,1} \cup I_{j,\,2})$, then
$b(y) = \al_{I_j}(b)$ and so $f_j(y) \lf[b(y) - \al_{I_j}(b)\r] = 0$. Thus inequality~\eqref{fj proper-1} holds for all~$y
\in I_j$.

To see~\eqref{fj proper-3} we first note that
\begin{equation}\label{eq:thm2.19}
   \lf|f_j(y)\r|
   = |I_j|^{-1/p}\lf|f^1_j(y)-f^2_j(y)\r|
   \geq \frac{1}{2}|I_j|^{-1/p}
   \quad \text{for all } y \in I_{j,\,1}\cup I_{j,\,2}.
\end{equation}
Second, for all $y \in I_{j,\,1}\cup I_{j,\,2}$ we also have
\begin{eqnarray*}
  \lf|f^1_j(y)-f^2_j(y)\r|
  %&=& \lf| \chi_{I_{j,\,1}}(y)-\chi_{I_{j,\,2}}(y) - a_j\chi_{I_j}(y)\r| \\
  \leq \lf|\chi_{I_{j,\,1}}(y)\r| + \lf|\chi_{I_{j,\,2}}(y)\r| + \lf|a_j\chi_{I_j}(y)\r|
  \leq  \frac{5}{2}.
   %&\leq& 1+1+\frac{1}{2}
\end{eqnarray*}
Thus
\begin{equation}\label{eq:thm2.20}
   \lf|f_j(y)\r|
   = |I_j|^{-1/p}\lf|f^1_j(y)-f^2_j(y)\r|
   \leq \frac{5}{2}|I_j|^{-1/p}
   \quad \text{for all } y \in I_{j,\,1}\cup I_{j,\,2}.
\end{equation}
So from inequalities~\eqref{eq:thm2.19} and~\eqref{eq:thm2.20}
we obtain the equivalence in~\eqref{fj proper-3}.

  %%%%%%%%%%%%%%%%%%%%%%%%%%%%%%%%%%%%%%%%%%%%%%%%%%%%%%%%%%%%%%%%%%%%%%%
Now, to see~\eqref{lower upper lpbdd C comm0} in Lemma~\ref{lem:vmo-contra}, using~\eqref{eq:thm2.10}, \eqref{eq:thm2.4} and~\eqref{fj proper-3} we compute
\[\lf\|f_j\r\|_{\lpz}^p  = \int_{I_{j,\,1}\cup I_{j,\,2}} \lf|f_j(x)\r|^p\,dx + \int_{\R \setminus (I_{j,\,1}\cup I_{j,\,2})} \lf|f_j(x)\r|^p\,dx
\ls 1 + \frac{1}{2^p} \ls 1,\]
as required.
%%%%%%%%%%%%%%%%%%%%%%%%%%%%%%%%%%%%%%%%%%%%%%%%%%%%%%%%%%%%%%%%%%%%%%%%%%

Next, fix a constant $A_1>4$.
Then for any integer $k\ge\lfloor\log_2 A_1\rfloor$, we claim that
\begin{eqnarray}\label{ijk set inclu}
2^{k+1}I_j\subset8I_j^k&=&\lf(x_j-\frac{5}{2}\cdot2^kr_j,\,x_j+\frac{11}{2}\cdot2^kr_j\r)\subset2^{k+3}I_j.
\end{eqnarray}
To see the first inclusion, we recall that $I_j =
I\lf(x_j,r_j\r) = \lf(x_j-r_j, x_j+r_j\r).$ Hence
\begin{eqnarray}
  2^{k+1}I_j &=& I\lf(x_j,2^{k+1}r_j\r)
  = \lf(x_j-2^{k+1}r_j, x_j+2^{k+1}r_j\r),
  \quad\text{and}
  \label{eq:thm2.21} \\
  2^{k+3}I_j &=& I\lf(x_j,2^{k+3}r_j\r)
  = \lf(x_j-4\cdot 2^{k+1}r_j, x_j+4\cdot 2^{k+1}r_j\r).
  \label{eq:thm2.23}
\end{eqnarray}
Also, as defined in Lemma~\ref{lem:vmo-contra},
\[
    I_j^k
    :=\lf(x_j+2^kr_j,\,x_j+2^{k+1}r_j\r)
    = I\lf(x_j + 3\cdot2^{k-1}r_j, 2^{k-1}r_j\r),
\]
and so
\begin{eqnarray}\label{eq:thm2.22}
   8I_j^k %&=& \noz I\lf(x_j + 3\cdot 2^{k-1}r_j, 8\cdot 2^{k-1}r_j\r)\noz\\
   &=& \lf(x_j - 5\cdot 2^{k-1}r_j, x_j + 11\cdot 2^{k-1}r_j\r)\noz\\
   %&=& \lf(x_j - \frac{5}{2}\cdot2^{k}r_j, x_j + \frac{11}{2}\cdot2^{k}r_j\r)\noz\\
   &=& \lf(x_j - \frac{5}{4}\cdot2^{k+1}r_j, x_j + \frac{11}{4}\cdot2^{k+1}r_j\r).
\end{eqnarray}
The inclusions in~\eqref{ijk set inclu} follow from equations~\eqref{eq:thm2.21}--\eqref{eq:thm2.22}, since
\[x_j-4\cdot2^{k+}r_j \leq x_j - \frac{5}{4}\cdot2^{k+1}r_j \leq x_j-2^{k+1}r_j, \text{ and}\]
\[x_j+2^{k+1}r_j \leq x_j + \frac{11}{4}\cdot2^{k+1}r_j \leq x_j+4\cdot2^{k+1}r_j.\]
%%%%%%%%%%%%%%%%%%%%%%%%%%%%%%%%%%%%%%%%%%%%%%%%%%%%%%%%%%

We turn to inequality \eqref{lower upper lpbdd C comm} in Lemma~\ref{lem:vmo-contra}.
Observe that
  \begin{equation}\label{com equiv}
\lf|[b, \Clm]f_j\r|=\bigg|\underbrace{\Clm\lf([b-\al_{I_j}(b)]f_j\r)}_{A(\cdot )}
-\underbrace{\lf[b-\al_{I_j}(b)\r]\Clm(f_j)}_{B(\cdot )}\bigg|.
\end{equation}
Using~\eqref{com equiv} and Minkowski's inequality for the~$L^p(I_j^k)$ norm, we have
\begin{align}\label{eq:thm2.28}
\lf\|\lf[b, \Clm\r]f_j\r\|_{L^p(I_j^k)} &= \lf\|A(\cdot )-B(\cdot )\r\|_{L^p(I_j^k)} \noz \\
&\geq \lf\|A(\cdot )\r\|_{L^p(I_j^k)} - \lf\|B(\cdot )\r\|_{L^p(I_j^k)}.
%&= \lf(\int_{I_j^k}\lf|A\r|^p\,dy\r)^{1/p} - \lf(\int_{I_j^k}\lf|B\r|^p\,dy\r)^{1/p}.
\end{align}
We will estimate the~$L^p$-norms of~$A$ and~$B$
in~\eqref{eq:thm2.28}.

We start with~$\|B(\cdot)\|_{L^p(I_j^k)}$. Note that
$|z-x_j|<\frac{1}{2}|y-x_j|$ for any~$z\in I_j$
and~$y\in\mathbb R\setminus 2I_j$  .
%This is because for all~$y\in\mathbb R\setminus 2I_j$ and~$z\in I_j$ we have
%\[\lf|z-x_j\r|\leq r_j \Rightarrow 2\lf|z-x_j\r|\leq 2r_j \leq \lf|y-x_j\r|.\]
Also, recall that the kernel~$\Clm(x,y)$ of the Cauchy integral is standard.
%Therefore the kernel
%\[\Clm(x,y) = \frac{1}{y-x +i\lf(A(y) - A(x)\r)}\]
%Therefore it satisfies
%\begin{equation}\label{eq:thm2.6}
%  |\Clm(x,y)|\leq \frac{1}{|y-x|},
%\end{equation}
%and for all~$y, x_j, z$ such that~$2|z-x_j| \leq |y-x_j|$ we have
%\begin{equation}\label{eq:thm2.7}
%  |\Clm(y,x_j)-\Clm(y,z)|\leq 2(L+1) \frac{|x_j-z|}{|x_j-y|^2}.
%\end{equation}
Using the fact that~$\supp(f_j)\subset I_j$, equations~\eqref{fj proper-2},~\eqref{eq:thm1.2} and~\eqref{fj proper-3}, we see that for all~$y\in\mathbb R\setminus 2I_j$ and~$z \in I_j$,
 \begin{align}\label{upper bdd riesz ope}
|B(y)|&=\lf|\lf[b(y)-\al_{I_j}(b)\r]\Clm(f_j)(y)\r|\noz \\
%&=&\lf|b(y)-\al_{I_j}(b)\r|\lf|\int_{I_j}\Clm(y,z)f_j(z)\,dz\r|\noz\\
%&=&\lf|b(y)-\al_{I_j}(b)\r|\lf|\int_{I_j}\lf[\Clm(y,z)-\Clm(y,x_j)\r]f_j(z)\,dz\r|\noz\\
&\le\lf|b(y)-\al_{I_j}(b)\r|\int_{I_j}|\Clm(y,z)-\Clm(y,x_j)||f_j(z)|\,dz\noz\\
&\ls\lf|b(y)-\al_{I_j}(b)\r|\int_{I_j} \frac{|x_j-z|}{|x_j-y|^2}|I_j|^{-1/p}\,dz\noz\\
&=\frac{\lf|b(y)-\al_{I_j}(b)\r|}{|I_j|^{1/p}|x_j-y|^2}\int_{I_j} |x_j-z|\,dz\noz\\
&\leq\frac{\lf|b(y)-\al_{I_j}(b)\r|}{|I_j|^{1/p}|x_j-y|^2}\int_{I_j} r_j\,dz\noz\\
%&=&\lf|b(y)-\al_{I_j}(b)\r|}{|I_j)|^{1/p}|x_j-y|^2}\lf|I_j\r| r_j\noz\\
%&=&2(L+1)r_j|I_j|^{1/p'}\frac{\lf|b(y)-\al_{I_j}(b)\r|}{|x_j-y|^2} \noz\\
&= r_j\lf|I_j\r|^{1/p'}\frac{|b(y)-\al_{I_j}(b)|}{|x_j-y|^2} \noz \\
&= 2^{-1}\lf|I_j\r|^{1+1/p'}\frac{|b(y)-\al_{I_j}(b)|}{|x_j-y|^2}.
\end{align}
Note that~$I_j^k = \lf(x_j+2^kr_j,x_j+2^{k+1}r_j\r) \subset
\lf(\R \setminus 2I_j\r)$ for all~$k \geq \lfloor\log_2
A_1\rfloor$ and $A_1 > 4$. Also, for all~$y \in I_j^k$, we have
\[ |x_j-y| \geq 2^kr_j =2^{k-1}|I_j| .\]
Thus by~\eqref{upper bdd riesz ope} we get
%%%%%%%%%%%%%%%%%%%%%%%%%%%%%%%%%%%%%%%%%%%%%%
%Hence,%there exists a positive constant $C_4$, such that for any $k\in \N$,
\begin{equation}\label{upper bdd com}
  \lf\|B(\cdot )\r\|_{L^p(I_j^k)} \ls \frac{2^{-1}\lf|I_j\r|^{1+1/p'}}{2^{2(k-1)}|I_j|^2}\lf\|b-\al_{I_j}(b)\r\|_{L^p(I_j^k)}
  = \frac{2\lf|I_j\r|^{-1/p}}{2^{2k}}\lf\|b-\al_{I_j}(b)\r\|_{L^p(I_j^k)}.
\end{equation}
We consider~$\lf\|b-\al_{I_j}(b)\r\|_{L^p(I_j^k)}$. Note that
for all~$k \geq \lfloor\log_2 A_1\rfloor$ and $A_1 > 4$ we have
\[I_j^k = \lf(x_j+2^kr_j,x_j+2^{k+1}r_j\r) \subset \lf(x_j-2^{k+1}r_j, x_j+2^{k+1}r_j\r) = 2^{k+1}I_j.\]
Thus we obtain
\begin{align}\label{eq:thm2.24}
  \lefteqn{\lf\|b-\al_{I_j}(b)\r\|_{L^p(I_j^k)}}\noz\\
  % &\leq& \int_{2^{k+1}I_j}|b(y)-\al_{I_j}(b)|^p\,dy \noz\\
  &\leq \lf\|b-\al_{2^{k+1}I_j}(b) + \al_{2^{k+1}I_j}(b)-\al_{I_j}(b)\r\|_{L^p(2^{k+1}I_j)} \noz\\
   %&\leq& \int_{2^{k+1}I_j}2^{p-1}\lf(|b(y)-\al_{2^{k+1}I_j}(b)|^p + |\al_{2^{k+1}I_j}(b)-\al_{I_j}(b)|^p\r)\,dy  \noz\\
   %&\leq& 2^{p-1}\lf(\int_{2^{k+1}I_j}|b(y)-\al_{2^{k+1}I_j}(b)|^p\,dy + \int_{2^{k+1}I_j}|\al_{2^{k+1}I_j}(b)-\al_{I_j}(b)|^p\,dy  \r) \noz\\
   &\leq \lf\|b-\al_{2^{k+1}I_j}(b) \r\|_{L^p(2^{k+1}I_j)} + \lf\|\al_{2^{k+1}I_j}(b)-\al_{I_j}(b)\r\|_{L^p(2^{k+1}I_j)}.
\end{align}
For the first term in the last line of~\eqref{eq:thm2.24}, using equation~\eqref{eq:intro2}, for every interval~$I$ we have
\[\int_{I}|b(y)-\al_{I}(b)|^p\,dy \ls |I|\|b\|_{\bmo}^p \ls |I|.\]
%\begin{eqnarray*}
%% \nonumber to remove numbering (before each equation)
%  \lf(\frac{1}{|I|}\int_{I}|b(y)-b_{I}|^p\,dy\r)^{1/p} &\ls& \|b\|_{\bmo} \\
%  \therefore \int_{I}|b(y)-b_{I}|^p\,dy &\ls& |I|\|b\|_{\bmo}^p \ls |I|.
%\end{eqnarray*}
Thus the first term in the last line of~\eqref{eq:thm2.24} is
controlled by~$2^{(k+1)/p}|I_j|^{1/p}$:
\begin{equation}\label{eq:thm2.25}
 \lf\|b-\al_{2^{k+1}I_j}(b) \r\|_{L^p(2^{k+1}I_j)} \ls  2^{(k+1)/p}|I_j|^{1/p}.
\end{equation}
For the second term in the last line of~\eqref{eq:thm2.24}, using equation~\eqref{eq:intro2} we have
 \begin{align*}
   |\al_{2^{k+1}I_j}(b)-\al_{I_j}(b)| &= \intav_{I_j}|\al_{2^{k+1}I_j}(b)-\al_{I_j}(b)|\,dy \\
   % &\leq& \intav_{I_j}|\al_{I_j}(b) -b(y)| + |b(y)- \al_{2^2I_j}(b)|\,dy \\
    &\leq  \intav_{2^{k+1}I_j}|\al_{2^{k+1}I_j}(b)) -b(y)|\,dy + \intav_{I_j} |b(y)- \al_{I_j}(b)|\,dy \\
    %&\ls& \|b\|_{\bmo} + \|b\|_{\bmo}  \\
    &\ls \|b\|_{\bmo} = 1.
 \end{align*}
 As a result
\begin{align}\label{eq:thm2.27}
    \lf\|\al_{2^{k+1}I_j}(b)-\al_{I_j}(b)\r\|_{L^p(2^{k+1}I_j)}
    &= \lf|2^{k+1}I_j\r|^{1/p}\lf|\al_{2^{k+1}I_j}(b)-\al_{I_j}(b)\r| \nonumber\\
    &\ls 2^{(k+1)/p}\lf|I_j\r|^{1/p}.
\end{align}
% \begin{eqnarray}\label{eq:thm2.27}
% % \nonumber to remove numbering (before each equation)
%   |\al_{2^{k+1}I_j}(b)-\al_{I_j}(b)| &\ls& k\|b\|_{\bmo} \ls k \noz\\
%   %\therefore  |\al_{2^{k+1}I_j}(b)-\al_{I_j}(b)|^p &\ls& k^p\noz\\
%    2^{k+1}|I_j||\al_{2^{k+1}I_j}(b)-\al_{I_j}(b)|^p &\ls& 2^{k+1}|I_j|k^p.
% \end{eqnarray}
Using~\eqref{eq:thm2.25} and~\eqref{eq:thm2.27} we can estimate
the left-hand side of~\eqref{eq:thm2.24} by
\begin{eqnarray*}
  \lf\|b-\al_{I_j}(b)\r\|_{L^p(I_j^k)}
  %&\leq& 2^{p-1}\lf(\int_{2^{k+1}I_j}|b(y)-\al_{2^{k+1}I_j}(b)|^p\,dy + 2^{k+1}|I_j||\al_{2^{k+1}I_j}(b)-\al_{I_j}(b)|^p \r) \\
   &\ls&  2^{(k+1)/p}|I_j|^{1/p} + 2^{(k+1)/p}\lf|I_j\r|^{1/p}
   %&=& (k^p+1)2^{p-1}2^{k+1}|I_j|\\
  \ls 2^{(k+1)/p}\lf|I_j\r|^{1/p}.
\end{eqnarray*}
Consequently, we can now estimate~\eqref{upper bdd com}:
\[\lf\|B(\cdot )\r\|_{L^p(I_j^k)}
\ls\frac{2\lf|I_j\r|^{-1/p}}{2^{2k}} 2^{(k+1)/p}\lf|I_j\r|^{1/p}
= C_4 \frac{1}{2^k} \frac{1}{2^{k(p-1)/p}}, \]
%\begin{eqnarray*}
%% \nonumber to remove numbering (before each equation)
%  \lf\|B(\cdot )\r\|_{L^p(I_j^k)} &\ls& \frac{2\lf|I_j\r|^{-1/p}}{2^{2k}} 2^{\frac{k+1}{p}}\lf|I_j\r|^{1/p} \\
%   &=& 2^{1+1/p} \frac{1}{2^k}\frac{1}{2^k 2^{-k/p}}  \\
%   &=&  C_4 \frac{1}{2^k} \frac{1}{2^{\frac{k(p-1)}{p}}},
%\end{eqnarray*}
where~$C_4 = C_4(p)$ is independent of~$k$ and~$j$.
%%%%%%%%%%%%%%%%%%%%%%%%%%%%%%%%%%%%%%%%%%%%%%%%%%%%%%%%%%%%%%%%%%%%%%%%%%%%%%%%%%%%

Next, we will estimate~$\|A(\cdot )\|_{L^p(I_j^k)}$.   %$\int_{I_j^k}|A|^p\,dy$.
Observe that for all~$y \in I_j^k$ and~$z \in I_j$ we have~$y>z$ and
\[|y-z|\leq |x_j + 2^{k+1}r_j - (x_j - r_j)| = (2^{k+1} + 1)r_j \leq 2^{k+2} r_j.\]
Using \eqref{eq:l1},  \eqref{fj proper-1}, \eqref{fj proper-3},
\eqref{eq:thm2.3} and  \eqref{lower bdd osci}, for all~$y \in I_j^k$ and~$z \in I_j$, we deduce a lower bound for~$|A(y)|$:
 \begin{align*}
|A(y)| %&=&\lf|\Clm\lf[(b-\al_{I_j}(b))f_j\r](y)\r|\\
&=\lf|\int_{(I_{j,\,1}\cup I_{j,\,2})}\Clm(y,z)\lf[b(z)-\al_{I_j}(b)\r]f_j(z)\,dz\r|\\
%&\lf|\int_{(I_{j,\,1}\cup I_{j,\,2})}\Clm(y,z)\lf[b(z)-\al_{I_j}(b)\r]f_j(z)\,dz\r|\\
&=\lf| \int_{(I_{j,\,1}\cup I_{j,\,2})} \frac{1}{z-y + i(A(z) - A(y))} \lf[b(z)-\al_{I_j}(b)\r] f_j(z)\,dz  \r|\\
&=\lf| \int_{(I_{j,\,1}\cup I_{j,\,2})} \frac{z-y}{(z-y)^2 + (A(z) - A(y))^2} \lf[b(z)-\al_{I_j}(b)\r] f_j(z)\,dz \right. \\
&\quad\left. \quad -i \int_{(I_{j,\,1}\cup I_{j,\,2})} \frac{A(z)-A(y)}{(z-y)^2 + (A(z) - A(y))^2} \lf[b(z)-\al_{I_j}(b)\r] f_j(z)\,dz \r|\\
&\geq \lf|\int_{(I_{j,\,1}\cup I_{j,\,2})} \frac{z-y}{(z-y)^2 + (A(z) - A(y))^2} \lf[b(z)-\al_{I_j}(b)\r] f_j(z)\,dz \right| \\
&\geq  \int_{(I_{j,\,1}\cup I_{j,\,2})} \frac{y-z}{(z-y)^2 + L^2(z-y)^2} \lf[b(z)-\al_{I_j}(b)\r] f_j(z)\,dz \\
%&=  \int_{(I_{j,\,1}\cup I_{j,\,2})} \frac{1}{(L^2+1)(y-z)} \lf[b(z)-\al_{I_j}(b)\r] f_j(z)\,dz  \\
&\gs  \int_{(I_{j,\,1}\cup I_{j,\,2})} \frac{1}{y-z} \lf[b(z)-\al_{I_j}(b)\r] f_j(z)\,dz  \\
%&=  \int_{(I_{j,\,1}\cup I_{j,\,2})} \frac{1}{|y-z|} \lf|\lf[b(z)-\al_{I_j}(b)\r] f_j(z)\r|\,dz  \\
&= \int_{(I_{j,\,1}\cup I_{j,\,2})} \frac{1}{|y-z|} \lf|b(z)-\al_{I_j}(b)\r| \lf|f_j(z)\r|\,dz.  \\
%&\leq&\int_{(I_{j,\,1}\cup I_{j,\,2})}|\Clm(y,z)|\lf|\lf[b(z)-\al_{I_j}(b)\r]\r|\lf|f_j(z)\r|\,dz\\
&\gs \int_{(I_{j,\,1}\cup I_{j,\,2})} \frac{1}{|y-z|}\lf|b(z)-\al_{I_j}(b)\r| \lf|I_j\r|^{-1/p} \,dz\\
%&\geq& \frac{\lf|I_j\r|^{-1/p}}{2^{k+2}r_j} \int_{(I_{j,\,1}\cup I_{j,\,2})} \lf|b(z)-\al_{I_j}(b)\r| \,dz \\
&\geq\frac{\lf|I_j\r|^{-1/p}}{2^{k+2}r_j} \int_{I_{j}} \lf|b(z)-\al_{I_j}(b)\r| \,dz \\
&\gs\frac{\lf|I_j\r|^{-1/p}}{2^{k+2}r_j} M(b,I_j) |I_j|\\
&> \frac{\e|I_j|^{1-1/p}}{2^{k+2}r_j} = \frac{\e}{2^{k+1}}|I_j|^{-1/p}.
\end{align*}
Consequently,
\[\|A(\cdot )\|_{L^p(I_j^k)} \gs \frac{\e}{2^{k+1}}|I_j|^{-1/p}\lf|I_j^k\r|^{1/p}
= \e2^{\frac{-1-p}{p}} \frac{1}{2^{k(p-1)/p}} %\frac{|I_j|^{\frac{p-1}{p}}}{|2^kI_j|^{\frac{p-1}{p}}}
= C_5 \e \frac{1}{2^{k(p-1)/p}},\]
where~$C_5=C_5(p)$ is independent of~$k$ and~$j$.

%%%%%%%%%%%%%%%%%%%%%%%%%%%%%%%%%%%%%%%%%

Therefore, returning to~\eqref{eq:thm2.28}, we have
\begin{eqnarray*}
   \lf\|\lf[b, \Clm\r]f_j\r\|_{L^p(I_j^k)} %&\geq&  \frac{1}{2^{p-1}}\int_{I_j^k}\lf|A\r|^p\,dy - \int_{I_j^k}\lf|B\r|^p\,dy \\
   &\gs&  C_5 \e \frac{1}{2^{k(p-1)/p}} - C_4\frac{1}{2^{k}} \frac{1}{2^{k(p-1)/p}}\\
   &=&  \lf( C_5 \e  - C_4\frac{1}{2^{k}}\r)  \frac{1}{2^{k(p-1)/p}}.
\end{eqnarray*}
Take $A_1$ large enough that for any integer $k\ge \lfloor\log_2 A_1\rfloor$,
%$$C_5\frac{\e^p}{2^{p-1}}-C_4k^p\ge C_5\frac{\e^p}{2^p}.$$
\[ C_5 \e  - C_4\frac{1}{2^{k}} \geq C_5 \frac{\e}{2}.\]
Then for all such~$k$ we have
\[ \lf\|\lf[b, \Clm\r]f_j\r\|_{L^p(I_j^k)} \gs C_5 \frac{\e}{2}  \frac{1}{2^{k(p-1)/p}}.\]
It follows that
\[ \int_{I_j^k}\lf|\lf[b, \Clm\r]f_j(y)\r|^p\,dy \gs C_5^p\frac{\e^p}{2^p}\frac{1}{2^{k(p-1)}}
 = \widetilde{C}_1\e^p\frac{1}{2^{k(p-1)}}.\]
This shows the  inequality \eqref{lower upper lpbdd C comm}.
%%%%%%%%%%%%%%%%%%%%%%%%%%%%%%%%%%%%%%%%%%%%%%%%%%%%%%%%%%%%%%%%%%%%%%%%%%%%%%%%%%%%%%%%%%%%%%%%

Finally we show the inequality \eqref{lower upper lpbdd C comm2} in Lemma~\ref{lem:vmo-contra}.
Using equation~\eqref{com equiv} we have
\begin{eqnarray}\label{eq:thm2.47}
  \lf\|[b, \Clm]f_j\r\|_{L^p(2^{k+1}I_j\setminus 2^kI_j)} &=&  \lf\|A(\cdot )
  -B(\cdot )\r\|_{L^p(2^{k+1}I_j\setminus 2^kI_j)} \noz \\
   &\leq&  \lf\|A(\cdot )\r\|_{L^p(2^{k+1}I_j\setminus 2^kI_j)} + \lf\|B(\cdot)\r\|_{L^p(2^{k+1}I_j\setminus 2^kI_j)}.
\end{eqnarray}
Consider the term~$ \lf\|A(\cdot )\r\|_{L^p(2^{k+1}I_j\setminus
2^kI_j)}$ in~\eqref{eq:thm2.47}. Note that for all~$z\in I_j$
and $y \in \R \setminus 2^k I_j$, we have
\[|y-z| \geq |x_j+2^kr_j - x_j -r_j| = (2^k-1)r_j.\]
Using the fact that $\supp(f_j)\subset I_j$, together with  \eqref{eq:thm1.1}, \eqref{fj proper-3} and~\eqref{eq:intro2},
we deduce that for all~$z\in I_j$ and $y\in \R\setminus 2^kI_j$,
\begin{eqnarray}\label{eq:thm2.46}
  |A(y)|
  %&=&\lf|\Clm\lf[(b-\al_{I_j}(b))f_j\r](y)\r|\\
%&=&\lf|\int_{I_j}\Clm(y,z)\lf[b(z)-\al_{I_j}(b)\r]f_j(z)\,dz\r|\\
&\leq&\int_{I_j}|\Clm(y,z)|\lf|b(z)-\al_{I_j}(b)\r|\lf|f_j(z)\r|\,dz \noz\\
  &\ls& \int_{I_j} \frac{1}{|y-z|}\lf|b(z)-\al_{I_j}(b)\r| \lf|I_j\r|^{-1/p} \,dz \noz\\
  %&\leq& \int_{I_j}\frac{1}{(2^k-1)r_j} \lf|b(z)-\al_{I_j}(b)\r| \lf|I_j\r|^{-1/p} \,dz\\
    &\leq& \frac{\lf|I_j\r|^{-1/p}}{(2^k-1)r_j} \int_{I_j}\lf|b(z)-\al_{I_j}(b)\r|  \,dz  \noz\\
  &\ls& \frac{\lf|I_j\r|^{-1/p}}{(2^k-1)r_j} \|b\|_{\bmo}|I_j| \noz\\
   %&\ls&  \frac{\lf|I_j\r|^{1-1/p}}{(2^k-1)r_j} \\
   &\ls&  \frac{|I_j|^{-1/p}}{2^{k-2}}.
\end{eqnarray}
The upper bound for~$|A(y)|$ in~\eqref{eq:thm2.46} gives us
\begin{eqnarray}\label{eq:thm2.48}
  \lf\|A(\cdot )\r\|_{L^p(2^{k+1}I_j\setminus 2^kI_j)}
  &\ls& \frac{|I_j|^{-1/p}}{2^{k-2}} |2^{k+1}I_j|^{1/p}
  \ls 2^{1/p} \frac{1}{2^{k(p-1)/p}} \noz \\
  &=& C_6  \frac{1}{2^{k(p-1)/p}},
\end{eqnarray}
where~$C_6=C_6(p)$ is independent of~$k$ and~$j$.

Consider now the term~$ \lf\|B(\cdot )\r\|_{L^p(2^{k+1}I_j\setminus
2^kI_j)}$ in~\eqref{eq:thm2.47}. Following the same argument
for estimating~$ \lf\|B(\cdot )\r\|_{L^p(I_j^k)}$ above, we
obtain
\begin{equation}\label{eq:thm2.49}
   \lf\|B(\cdot )\r\|_{L^p(2^{k+1}I_j\setminus 2^kI_j)} \ls C_7 \frac{1}{2^k} \frac{1}{2^{k(p-1)/p}},
\end{equation}
where~$C_7=C_7(p)$ is independent of~$k$ and~$j$.
Using~\eqref{eq:thm2.48} and~\eqref{eq:thm2.49}, for all $k\ge\lfloor \log_2A_1\rfloor$, we have
\[\lf\|[b, \Clm]f_j\r\|_{L^p(2^{k+1}I_j\setminus 2^kI_j)} \ls \lf(C_6  + C_7 \frac{1}{2^k}\r)\frac{1}{2^{k(p-1)/p}}.\]
It follows that
\[\int_{2^{k+1}I_j\setminus 2^kI_j}\lf|[b, \Clm]f_j(y)\r|^p\,dy \leq \widetilde{C}_2\frac{1}{2^{k(p-1)}}, \]
which is~\eqref{lower upper lpbdd C comm2}, as required.
This completes the proof of Lemma~\ref{lem:vmo-contra}. \hfill\(\Box\)

%%%%%%%%%%%%%%%%%%%%%%%%%%%%%%%%%%%%%

With Lemma~\ref{lem:vmo-contra} in hand, we now return to the proof of our main result.

\noindent\emph{Proof of Theorem~\ref{thm2}.}
\emph{Sufficiency:}
%We use the idea in \cite{Uch78}.
We first show that if $[b, \Clm]$ is a compact operator on $\lpz$, then $b\in \vmo$.
Since $[b, \Clm]$ is compact on $\lpz$, $[b, \Clm]$ is bounded on $\lpz$.
Without loss of generality, we may assume that $\|b\|_{\bmo(\R)}=1$.
To show $b\in\vmo$, we use a contradiction argument via Definition~\ref{def:VMODafni}.
Observe that if $b\notin \vmo$, $b$ does not satisfy at least one of conditions~(1)--(3) in Definition~\ref{def:VMODafni}.
We consider the three cases separately.

{\sc Case 1:} Suppose $b$ does not satisfy condition~(1) in Definition~\ref{def:VMODafni}, that is,
\[\lim_{\dz \to 0} \sup_{I,|I|<\dz} \intav_I |f(x) - f_I| \,dx \neq 0.\]
 Then there exist $\e >0$
and a sequence $\{I_j\}_{j=1}^\fz$ of intervals satisfying
\[M(b,I_j) > \e \quad \text{for each }j\]
and~$|I_j| \to 0$ as~$j \to \fz$.
Let $f_j$, $\wz C_1$, $\wz C_2$, $A_1$ be as in Lemma \ref{lem:vmo-contra} and let $A_2>A_1$ be a large number to be chosen later.
%$$A_3:=8^{(1-p)(2\lz+1)}\wz C_1\dz^pA_1^{(1-p)(2\lz+1)}>
%\frac{2\wz C_2}{1-[\min(2^{2\lz},\,2)]^{1-p}}\frac1{[\min(2^{2\lz},\,2)]^{\lfloor\log_2 A_2\rfloor(p-1)}}.$$
Since $\lf|I_j\r|\to 0$ as $j\to\fz$, we may choose a subsequence $\{I_{j_\ell}^{(1)}\}$ of $\{I_j\}$ such that
\begin{equation}\label{descreasing interval}
\frac{\lf|I_{j_{\ell+1}}^{(1)}\r|}{\lf|I_{j_{\ell}}^{(1)}\r|}<\frac{1}{A_2} \qquad \text{for all } l\in \mathbb N.
\end{equation}
For fixed $\ell$, $m\in \mathbb N$, denote
\begin{eqnarray*}
  \mathcal J &:=& \lf(x_{j_\ell}^{(1)}+A_1r_{j_\ell}^{(1)}, x_{j_\ell}^{(1)}+A_2r_{j_\ell}^{(1)}\r), \\
  \mathcal J_1 &:=& \mathcal J\setminus \lf\{y\in\R: \lf|y-x_{j_{\ell+m}}^{(1)}\r|\le A_2r_{j_{\ell+m}}^{(1)}\r\}, \quad \text{ and}\\
  \mathcal J_2 &:=& \lf\{y\in\R: \lf|y-x_{j_{\ell+m}}^{(1)}\r|>A_2r_{j_{\ell+m}}^{(1)}\r\}.
\end{eqnarray*}
Note that
$$\mathcal J_1\subset\lf\{y\in\R: \lf|y-x_{j_{\ell}}^{(1)}\r|\le A_2r_{j_{\ell}}^{(1)}\r\}\cap\mathcal J_2
\,\,{\rm and}\,\, \mathcal J_1= \mathcal J \cap \mathcal J_2 = \mathcal J\setminus(\mathcal J\setminus\mathcal J_2).$$
 We then have
\begin{eqnarray}\label{low lpbdd comparing com}
\lefteqn{\|\lf[b,\Clm\r](f_{j_\ell})-\lf[b,\Clm\r](f_{j_{\ell+m}})\|_{L^p(\R)}}\noz\\
&\ge&|\lf[b,\Clm\r](f_{j_\ell})-\lf[b,\Clm\r](f_{j_{\ell+m}})\|_{L^p(\mathcal J_1)}\noz\\
&\ge&\|\lf[b,\Clm\r](f_{j_\ell})\|_{L^p(\mathcal J_1)}
-\|\lf[b,\Clm\r](f_{j_{\ell+m}})\|_{L^p(\mathcal J_1)}\noz\\
&\ge&\|\lf[b,\Clm\r](f_{j_\ell})\|_{L^p(\mathcal J_1)}
-\|\lf[b,\Clm\r](f_{j_{\ell+m}})\|_{L^p(\mathcal J_2)}\noz\\
&=&\lf(\int_{\mathcal J\setminus(\mathcal J\setminus\mathcal J_2)}\lf|\lf[b,\Clm\r](f_{j_\ell})(y)\r|^p\,dy\r)^{1/p}
-\lf(\int_{\mathcal J_2}\lf|\lf[b,\Clm\r](f_{j_{\ell+m}})(y)\r|^p\,dy\r)^{1/p}\noz\\
&=:&{\rm F_1}-{\rm F_2}. \noz
\end{eqnarray}
We first consider the term ${\rm F_1}$. To begin with, we estimate the measure of the set~$E_{j_\ell}:=\mathcal J\setminus\mathcal J_2$.  Assume that $E_{j_\ell}\not=\emptyset$.
Then $E_{j_\ell}\subset A_2I^{(1)}_{j_{\ell+m}}$.
Hence, we have
\begin{eqnarray}\label{ee1}
\lf|E_{j_\ell}\r|\le \lf|A_2I^{(1)}_{j_{\ell+m}}\r| = A_2\lf|I^{(1)}_{j_{\ell+m}}\r|<\lf|I^{(1)}_{j_{\ell}}\r|,
\end{eqnarray}
where the last inequality follows from \eqref{descreasing interval}.

Now for each~$k \geq \lfloor\log_2 A_1\rfloor$, as in Lemma~\ref{lem:vmo-contra} let
$$I_{j_\ell}^{k}:=\lf(x^{(1)}_{j_\ell}+2^kr^{(1)}_{j_\ell},\,x^{(1)}_{j_\ell}+2^{k+1}r^{(1)}_{j_\ell}\r).$$
Then
\begin{eqnarray}\label{ee2}
 \lf|I_{j_\ell}^{k}\r| &=& 2^kr^{(1)}_{j_\ell} =2^k\frac{\lf|I_{j_\ell}^{(1)}\r|}{2} \\
  &=& 2^{k-1}\lf|I_{j_\ell}^{(1)}\r| \geq  \lf|I_{j_\ell}^{(1)}\r| > \lf|E_{j_\ell}\r| \quad \text{for all } k \geq \lfloor\log_2 A_1\rfloor. \noz
\end{eqnarray}
Notice also that by definition,
\begin{equation}\label{ee3}
  E_{j_\ell} \subset \mathcal{J} \subset \bigcup_{k =  \lfloor\log_2 A_1\rfloor}^{\infty} I_{j_\ell}^{k}.
\end{equation}
Here  the second inclusion holds because the left endpoint of~$I_{j_\ell}^{\lfloor\log_2 A_1\rfloor}$ is~$x^{(1)}_{j_\ell}+2^{ \lfloor\log_2 A_1\rfloor}r^{(1)}_{j_\ell}$, which lies to the left of the left endpoint of~$\mathcal J$.

From inequality~\eqref{ee2} and the fact~\eqref{ee3}, it follows that~$E_{j_\ell}$ is covered by the union of at most two (adjacent) intervals~$I_{j_\ell}^{k}$.
That is, there is some~$k_0 \geq \lfloor\log_2 A_1\rfloor$ such that
$E_{j_\ell}\subset (I_{j_\ell}^{k_0}\cup I_{j_\ell}^{k_0+1})$.
By inequality~\eqref{lower upper lpbdd C comm} in Lemma~\ref{lem:vmo-contra},
\begin{eqnarray}\label{eq:thm2.29}
{\rm F}_1^p %&&= \int_{\mathcal J\setminus(\mathcal J\setminus\mathcal J_2)}\lf|\lf[b,\Clm\r](f_{j_\ell})(y)\r|^p\,dy \noz\\
&&\ge\sum_{k=\lfloor\log_2 A_1\rfloor+1,\,k\not=k_0,\,k_0+1}^{\lfloor\log_2 A_2\rfloor}
\int_{I_{j_\ell}^k}\lf|\lf[b,\Clm\r](f_{j_\ell})(y)\r|^p\,dy\noz\\
&&\ge\wz C_1\e^p\sum_{k=\lfloor\log_2 A_1\rfloor+1,\,k\not=k_0,\,k_0+1}^{\lfloor\log_2 A_2\rfloor}
\frac{1}{2^{k(p-1)}}\noz\\
&&\ge\wz C_1\e^p\sum_{k=\lfloor\log_2 A_1\rfloor+3}^{\lfloor\log_2 A_2\rfloor}\frac{1}{2^{k(p-1)}}\noz\\
%&&\ge\wz C_1\e^p\frac{1}{2^{(\lfloor\log_2 A_1\rfloor+3)(p-1)}}\noz\\
&&\ge 8^{(1-p)}\wz C_1\e^pA_1^{(1-p)}=:A_3.
\end{eqnarray}
If  $E_{j_\ell}:=\mathcal J\setminus\mathcal J_2=\emptyset$, then inequality~\eqref{eq:thm2.29} still holds.

On the other hand, using \eqref{lower upper lpbdd C comm2} in Lemma~\ref{lem:vmo-contra}, we deduce that
\begin{eqnarray*}
{\rm F}_2^p %&&= \int_{\mathcal J_2}\lf|\lf[b,\Clm\r](f_{j_{\ell+m}})(y)\r|^p\,dy \\
&&\le\sum_{k=\lfloor\log_2 A_2\rfloor}^{\fz}
\int_{2^{k+1}I_{j_{\ell+m}}^{(1)} \setminus 2^{k}I_{j_{\ell+m}}^{(1)}}\lf|\lf[b,\Clm\r](f_{j_{\ell+m}})(y)\r|^p dy\\
&&\le\wz C_2\sum_{k=\lfloor\log_2 A_2\rfloor}^\fz\frac{1}{2^{k(p-1)}}\\
&&\le \wz C_2 \frac{1/2^{p-1}}{1- 1/2^{p-1}} \\
%&&= \wz C_2 \frac{1}{1- 2^{1-p}}\frac{1}{2^{p-1}}\\
&&\le \frac{\wz C_2}{1- 2^{1-p}}\frac{1}{2^{\lfloor\log_2 A_2\rfloor(p-1)}}.
\end{eqnarray*}
%The fourth inequality above holds because for~$|x| <1$,
%\[\sum_{k=1}^{\infty} x^k = \frac{x}{1-x}.\]
%For us
%\[x = \frac{1}{2^{p-1}} = 2^{1-p} <1 \quad \text{for all } p \in (1,\infty).\]
If we choose~$A_2 > A_1$ large enough such that
\begin{equation}\label{eq:thm2.45}
A_3:=8^{(1-p)}\wz C_1\e^pA_1^{(1-p)}>
\frac{2\wz C_2}{1-2^{1-p}}\frac1{2^{\lfloor\log_2 A_2\rfloor(p-1)}},
\end{equation}
then we have
\begin{equation}\label{eq:thm2.30}
  {\rm F}_2^p \leq \frac{A_3}{2}.
\end{equation}
%By \eqref{set I condition_1} and applying \eqref{upper lpbdd riesz comm2} for $E_j:=\mathcal J\setminus\mathcal J_2$, we get
%\begin{equation}\label{upper lpbdd riesz comm3}
%\int_{\mathcal J\setminus\mathcal J_2}\lf|\lf[b,\Clm\r](f_{j_\ell})(y)\r|^p\ytz\le A_3/4.
%\end{equation}
By inequalities~\eqref{eq:thm2.29} and~\eqref{eq:thm2.30}, we get
\begin{eqnarray*}
&&\|\lf[b,\Clm\r](f_{j_\ell})-\lf[b,\Clm\r](f_{j_{\ell+m}})\|_{L^p(\mathbb R)}\gs A_3^{1/p} >0.
\end{eqnarray*}
Thus, $\{[b,\Clm]f_j\}_{j}$ is not relatively compact in $L^p(\R)$, which implies that
$[b, \Clm]$ is not compact on $\lpz$. This contradiction implies that, $b$ satisfies condition (1) in Definition~\ref{def:VMODafni}.

%%%%%%%%%%%%%%%%%%%%%%%%%%%%%%%%%%%%%%%%%%%%%%%%%%%%%%%%%%%%%%%%%%%%%%%%%%%%%%%%%%%%%%%%%%%%%%%%%%%
 {\sc Case 2:} Suppose $b$ violates condition~(2) in Definition~\ref{def:VMODafni}, that is,
\[\lim_{R\rightarrow\infty}\sup_{I,|I|>R}\frac{1}{|I|}\int_I|f(x)-f_I|\,dx\neq 0.\]
 In this case, there exist
$\e >0$ and a sequence $\{I_j\}$ of intervals satisfying~$M(b,I_j) > \e$
and that $|I_j|\rightarrow\infty$ as $j\rightarrow\infty$. We take a subsequence $\{I_{j_\ell}^{(2)}\}$ of $\{I_j\}$ such that
\begin{equation}\label{increasing interval}
\frac{\lf|I_{j_{\ell}}^{(2)}\r|}{\lf|I_{j_{\ell+1}}^{(2)}\r|}<\frac1{A_2} \qquad \text{for all } l\in \mathbb N,
\end{equation}
where~$A_2$ is chosen as in Case~1 above. We use a similar method to that in the previous case, but redefine our sets with the roles of~$j_{\ell}$ and~$j_{\ell+m}$ reversed. That is, for fixed $\ell$ and $m$, let
\begin{eqnarray*}
  \widetilde{\mathcal J} &:=& \lf(x_{j_{\ell+m}}^{(2)}+A_1r_{j_{\ell+m}}^{(2)}, x_{j_{\ell+m}}^{(2)}+A_2r_{j_{\ell+m}}^{(2)}\r), \\
  \widetilde{\mathcal J}_1 &:=& \widetilde{\mathcal J}\setminus \lf\{y\in\mathbb R: \lf|y-x_{j_{\ell}}^{(2)}\r|\le A_2r_{j_{\ell}}^{(2)}\r\}, \quad \text{ and}\\
  \widetilde{\mathcal J}_2 &:=& \lf\{y\in\mathbb R: \lf|y-x_{j_{\ell}}^{(2)}\r|>A_2r_{j_{\ell}}^{(2)}\r\}.
\end{eqnarray*}
Then we have that
$$\widetilde{\mathcal J}_1\subset\lf\{y\in\mathbb R: \lf|y-x_{j_{\ell+m}}^{(2)}\r|\le A_2r_{j_{\ell+m}}^{(2)}\r\}\cap\widetilde{\mathcal J}_2
\,\, {\rm and}\,\,
\widetilde{\mathcal J}_1=\widetilde{\mathcal J} \cap \widetilde{\mathcal J}_2 = \widetilde{\mathcal J}\setminus\lf(\widetilde{\mathcal J}\setminus\widetilde{\mathcal J}_2\r).$$
%%%%%%%%%%%%%%%%%%%%%%%%%%%%%%%%%%%%%%%%%%%%%%%%%%%%%%%%%%%%%%%%%%%%%%%%%%%%%%%%%%%%
Consequently,
\begin{eqnarray}\label{low lpbdd comparing com 2}
\lefteqn{\|\lf[b,\Clm\r](f_{j_{\ell+m}}) - \lf[b,\Clm\r](f_{j_\ell})\|_{L^p(\R)}}\noz\\
&\geq&\lf(\int_{\widetilde{\mathcal J}\setminus(\widetilde{\mathcal J}\setminus\widetilde{\mathcal J}_2)}\lf|\lf[b,\Clm\r](f_{j_{\ell+m}})(y)\r|^p\,dy\r)^{1/p}
-\lf(\int_{\widetilde{\mathcal J}_2}\lf|\lf[b,\Clm\r](f_{j_{\ell}})(y)\r|^p\,dy\r)^{1/p}\noz\\
&=:&{\widetilde{\rm F}_1}-{\widetilde{\rm F}_2}. \noz
\end{eqnarray}
By inequalities~\eqref{lower upper lpbdd C comm} and~\eqref{lower upper lpbdd C comm2}  in Lemma~\ref{lem:vmo-contra} and the definition of~$A_3$ in~\eqref{eq:thm2.45}, we can deduce that~$ \widetilde{\rm F}^p_1 \geq A_3$ and~$\widetilde{\rm F}^p_2 \leq A_3/2$, just as~${\rm F}_1^p$ and ~${\rm F}^p_2$  in Case~1.
As a consequence,
\begin{eqnarray*}
&&\|\lf[b,\Clm\r](f_{j_{\ell+m}})-\lf[b,\Clm\r](f_{j_{\ell}})\|_{L^p(\mathbb R)}\gs(A_3)^{1/p}.
\end{eqnarray*}
As in Case 1, by Lemma \ref{lem:vmo-contra} and inequality~\eqref{increasing interval},
we see that $[b, \Clm]$ is not compact on $\lpz$. This contradiction implies that
$b$ satisfies condition~(2) of Definition~\ref{def:VMODafni}.
%%%%%%%%%%%%%%%%%%%%%%%%%%%%%%%%%%%%%%%%%%%%%%%%%%%%%%%%%%%%%%%%%%%%%%%%%%%%%%%%%%%%%%%%%%%%%5

{ \sc Case 3:} By Cases 1 and 2, we may assume that conditions~(1) and~(2) in Definition~\ref{def:VMODafni} hold for~$b$. Suppose  condition (3) in Definition~\ref{def:VMODafni} fails, that is,
 \[\lim_{R\rightarrow\infty}\sup_{I,I\cap I(0,R)=\emptyset}\frac{1}{|I|}\int_I|f(x)-f_I|\,dx \neq 0.\]
 Then there exist $\e>0$ such that for each $R>0$, there exists an interval $I$ such that~$I \cap (-R,R) = \emptyset$ with $M(b,\,I)>\e$.
We claim that for the $\e$ above, there exists a sequence $\{I_j\}_j$ of intervals such that for all~$j \in \mathbb N $,
\begin{equation}\label{lower bdd osci-2}
M(b,\,I_j)>\e,
\end{equation}
and that for all $\ell\neq m$, and for the constant~$A_2$ chosen in Case 1 above,
\begin{equation}\label{I-j-pro}
A_2I_{\ell}\cap A_2I_m=\emptyset.
\end{equation}
To see this, first note that as $b$ satisfies condition~(2) in Definition~\ref{def:VMODafni}, for the aforementioned~$\e$ there exists a constant $\widetilde{C}_\e$ such that $$M(b,\,I)<\e$$
for every interval $I$ satisfying $|I|>\widetilde{C}_\e$.
Let $C_\e := \widetilde{C}_\e/2$. Then for $R_1>C_\e$,
there exists an interval $I_1:=I(x_1,\,r_1)\subset \R \backslash I(0,R_1)$ such that \eqref{lower bdd osci-2} holds. Similarly, for $R_j:=\lf|x_{j-1}\r|+4A_2C_\e$, $j=2, 3, \ldots$ \,, there exists $I_j:=I(x_j,\,r_j) \subset \R \backslash I(0,R_j)$ satisfying \eqref{lower bdd osci-2}.
Repeating this procedure, we obtain a collection $\{I_j\}_j$ of intervals satisfying \eqref{lower bdd osci-2} for each $j$.

By the choice of $\{I_j\}$, namely $M(b,I_j)>\e$, we have that~$|I_j| \leq \widetilde{C}_\e$, and so~$r_j \leq \widetilde{C}_\e/2 =C_{\e}$ for all $j \in \mathbb{N}$.
Thus
\[A_2r_j < A_2C_\e < 4A_2C_\e.\]
Therefore for all~$\ell \neq m$ we have
 \begin{eqnarray*}
 d(A_2I_{\ell},\,A_2I_m) &\geq&  R_j - (x_{j-1}+A_2r_{j-1})= x_{j-1}+4A_2C_\delta - x_{j-1} - A_2r_{j-1} \\
    &\geq& 4A_2C_\delta - A_2C_\delta  = 3A_2C_\delta.
 \end{eqnarray*}
This establishes the claim.

Now we define
\begin{eqnarray*}
  \widehat{\mathcal J}_1 &:=& \lf(x_{\ell}+A_1r_{\ell}, x_{\ell}+A_2r_{\ell}\r), \quad \text{ and} \\
  \widehat{\mathcal J}_2 &:=& \lf\{y\in\mathbb R: \lf|y-x_{\ell+m}\r|> A_2r_{\ell+m}\r\}.
\end{eqnarray*}
Note that $\widehat{\mathcal J}_1\subset\widehat{\mathcal J}_2$.
Thus, similarly to the estimates of ${\rm F_1}$ and ${\rm F_2}$ in Case~1, for all $\ell$, $m \in \mathbb{N}$, we get
\begin{eqnarray*}
&&\|\lf[b,\Clm\r](f_{\ell})-\lf[b,\Clm\r](f_{\ell+m})\|_{L^p(\mathbb R)}\\
&&\quad\ge\lf\{\int_{\widehat{\mathcal J}_1}\lf|\lf[b,\Clm\r](f_{\ell})(y)-\lf[b,\Clm\r](f_{\ell+m})(y)\r|^p\,dy\r\}^{1/p}\\
&&\quad\ge\lf\{\int_{\widehat{\mathcal J}_1}\lf|\lf[b,\Clm\r](f_{\ell})(y)\r|^p\,dy\r\}^{1/p}-
\lf\{\int_{\widehat{\mathcal J}_2}\lf|\lf[b,\Clm\r](f_{\ell+m})(y)\r|^p\,dy\r\}^{1/p}\\
%&&\quad\gs\lf(A_3\r)^{1/p}\\
&&\quad=:{\widehat{\rm F}_1}-{\widehat{\rm F}_2}.
\end{eqnarray*}
Again, by~\eqref{lower upper lpbdd C comm} and~\eqref{lower upper lpbdd C comm2}  in Lemma~\ref{lem:vmo-contra} and the definition of~$A_3$ in~\eqref{eq:thm2.45}, we deduce that~$ \widehat{\rm F}^p_1 \geq A_3$ and~$\widehat{\rm F}^p_2 \leq A_3/2$, as for~${\rm F}_1^p$ and ~${\rm F}^p_2$  in Case~1.
As a result, we get
\begin{eqnarray*}
&&\|\lf[b,\Clm\r](f_{\ell})-\lf[b,\Clm\r](f_{\ell+m})\|_{L^p(\mathbb R)}\gs(A_3)^{1/p}.
\end{eqnarray*}
Thus, $\{[b,\Clm]f_\ell\}_{\ell}$ is not relatively compact in $L^p(\R)$, which implies that
$[b, \Clm]$ is not compact on $\lpz$.
This contradicts the compactness of $[b, \Clm]$ on $L^p(\mathbb R)$,
so $b$ satisfies condition (3)
in Definition~\ref{def:VMODafni}.
This completes the proof of the sufficiency in Theorem~\ref{thm2}.
%Similarly, we can prove if $b$ does not
%satisfy (ii) or (iii) of Theorem \ref{t-com char}(please check), then $[b, \Clm]$ is not a compact operator.
%%%%%%%%%%%%%%%%%%%%%%%%%%%%%%%%%%%%%%%%%%%%%%%%%%%%%%%%%%%%%%%%%%%%%%%%%%%%%%%%%%%%%%%%%%%%%%%%%%%%%%%%%%%%%%%%5

\textbf{Necessity:}
To see the converse,
%The main ingredient is the Fr\'{e}chet-Kolmogorov theorem.
we must show that when $b\in\vmo(\R)$, the commutator $[b, \Clm]$ is compact on $\lpz$.
By a density argument, it suffices to show that $[b, \Clm]$ is a compact operator for $b\in C_c^{\infty}(\R)$.

Let $b\in C_c^{\infty}(\R)$. To show $[b,\Clm]$ is compact on $\lpz$, it suffices to show that for every bounded subset $E\subset\lpz$, the set
$[b,\Clm]E$ is precompact. Thus, we only need to show that $[b,\Clm]E$ satisfies the
hypotheses (a)--(c) in the Frech\'{e}t--Kolmogorov theorem (Theorem \ref{t-fre kol}). We first point out that by Theorem~\ref{thm1}
and the fact that $b\in \bmo(\R)$, $[b,\Clm]$ is bounded on $\lpz$, which implies that $[b,\Clm]E$
satisfies hypothesis (a) in Theorem \ref{t-fre kol}.

Next, we show that~$[b,\Clm]E$ satisfies hypothesis (b) in Theorem \ref{t-fre kol}.
We may assume that~$b\in C_c^{\infty}(\R)$ with~$\supp b \subset I(0,R)$.
For~$t>2$, set~$K^c := \{x \in \R: |x| > tR\}.$
Then
\begin{eqnarray*}
  \|[b,\Clm]f(x)\|_{L^p(K^c)} %&=& \int_{|x|>tR}|[b,\Clm]f(x)|^p \,dx \\
   &=&  \|b\Clm(f)(x) -  \Clm(bf)(x)\|_{L^p(K^c)}\\
   &\leq&  \|b\Clm(f)(x)\|_{L^p(K^c)} +  \|\Clm(bf)(x)\|_{L^p(K^c)}.
\end{eqnarray*}
Since~$\supp b\cap K^c = \emptyset$,  we have
\[ \int_{|x|>tR}\lf|b\Clm(f)(x)\r|^p \,dx  = 0,\]
and so
\begin{equation}\label{eq:thm2.40}
\|[b,\Clm]f(x)\|_{L^p(K^c)} \leq \|\Clm(bf)(x)\|_{L^p(K^c)}.
\end{equation}
Using equation~\eqref{eq:thm1.1} and the fact that~$\supp b \subset I(0,R)$ we have
\begin{align}\label{eq:thm2.41}
  \lf|\Clm(bf)(x)\r| %&=& \lf|\int_{\R} \Clm(x,y) b(y)f(y) \,dy\r|\noz \\
   &\leq  \int_{|y|<R} |\Clm(x,y)|| b(y)||f(y)| \,dy\noz\\
   &\leq   \int_{|y|<R} \frac{1}{|x-y|}| b(y)||f(y)| \,dy.
\end{align}
Notice that for~$|x|>tR$, $t >2$ and~$|y|<R$ we have~$|x-y|>|x|/2.$
In view of this fact and H\"{o}lder's inequality, inequality~\eqref{eq:thm2.41} yields
\begin{align*}
  \lf|\Clm(bf)(x)\r| %&\leq& \int_{|y|<R} \frac{2}{|x|}| b(y)||f(y)| \,dy \\
   &\leq \frac{2}{|x|} \int_{|y|<R} | b(y)||f(y)| \,dy \\
   &\leq  \frac{2}{|x|} \lf(\int_{|y|<R} | b(y)|^{p'} \,dy\r)^{1/p'} \lf(\int_{|y|<R} | f(y)|^{p} \,dy\r)^{1/p}  \\
   %&\leq& \frac{2}{|x|} \|b\|_{L^{\infty}(\R)} \lf(\int_{|y|<R}  1 \,dy\r)^{1/p'} \|f\|_{L^{p}(\R)}\\
   &\leq \frac{2}{|x|} \|b\|_{L^{\infty}(\R)}\|f\|_{L^{p}(\R)} (2R)^{1/p'}  \\
   &=  2^{1+ 1/p'} \|b\|_{L^{\infty}(\R)}\|f\|_{L^{p}(\R)} R^{1/p'}  \frac{1 }{|x|},
\end{align*}
since~$b \in C_c^{\infty}(\R)$. With this estimate of~$\lf|\Clm(bf)(x)\r|$, inequality~\eqref{eq:thm2.40} becomes
\begin{align*}
  \|[b,\Clm]f(x)\|_{L^p(K^c)} &\leq 2^{1+ 1/p'} \|b\|_{L^{\infty}(\R)}\|f\|_{L^{p}(\R)} R^{1/p'} \lf( \int_{|x|>tR} \frac{1}{|x|^p} \,dx\r)^{1/p}\\
   &= \frac{2^{2+1/p'}}{(p-1)^{1/p}} \|b\|_{L^{\infty}(\R)}\|f\|_{L^{p}(\R)}  \frac{1}{t^{1/p'}} = C \frac{1}{t^{1/p'}}.
\end{align*}
Finally, given each~$\e > 0$, we can choose~$t$ large enough such that~$Ct^{-1/p'} < \e$. Here the constant~$C$ depends on~$b$ and on the bound on~$\|f\|_{L^{p}(\R)}$ for~$f \in E$.
%\[  \frac{1}{(t)^{1/p'}} < \e.\]
Hence hypothesis (b)  in Theorem \ref{t-fre kol} holds for $[b,\Clm]E$.

It remains to prove that $[b,\Clm]E$ also satisfies hypothesis (c) of Theorem \ref{t-fre kol}.
Let $\e$ be a fixed positive constant in $(0,1/2)$.
Since~$b \in C_c^{\infty}(\R)$, it is uniformly continuous.
Choose~$z_0 = z_0(b,\e)$ sufficiently small that for all~$z \in (0,z_0)$, we have both~$|z|<\e^2$ and for all~$x\in\R$, $|b(x) - b(x+z)| < \e$. Fix~$z \in (0,z_0)$.
Then for all~$x\in\R$,
\begin{eqnarray*}
&&[b, \Clm]f(x)-[b, \Clm]f(x+z)\\
&&\quad=\int_{\R} \Clm(x, y)[b(x)-b(y)] f(y)\,dy\\
&&\quad\quad-\int_{\R} \Clm(x+z, y)[b(x+z)-b(y)] f(y)\,dy\\
&&\quad=\int_{|x-y|>\e^{-1} |z|}\Clm(x, y)[b(x)-b(x+z)]f(y)\,dy\\
&&\quad\quad+\int_{|x-y|>\e^{-1} |z|}[\Clm(x, y)-\Clm(x+z,y)][b(x+z)-b(y)]f(y)\,dy\\
&&\quad\quad+\int_{|x-y|\le\e^{-1} |z|}\Clm(x,y)[b(x)-b(y)]f(y)\,dy\\
&&\quad\quad-\int_{|x-y|\le\e^{-1}|z|}\Clm(x+z,y)[b(x+z)-b(y)]f(y)\,dy\\
&&\quad=:\sum_{j=1}^4{\rm L}_i.
\end{eqnarray*}

We start with~${\rm L}_2$. Since $\e\in(0, 1/2)$, it follows that
\[|x-y|>\e^{-1}|z| \Rightarrow |(x+z)-x|  < \frac{|x-y|}{2}.\]
Thus we may apply the smoothness condition of the kernel~$\Clm(x,y)$, concluding that
 \[ |\Clm(x,y) - \Clm(x+z,y)| \leq \frac{2(L+1)|x+z-x|}{|y-x |^2} = \frac{2(L+1)|z|}{|y-x |^2}.\]
Using this inequality together with the fact that~$b\in C_c^{\infty}(\R)$,  we get
\begin{align*}
  |{\rm L}_2| %&\leq& \int_{|x-y|>\e^{-1} |z|}|\Clm(x, y)-\Clm(x+z,y)||b(x+z)-b(y)||f(y)|\,dy \\
   %&\leq& \int_{|x-y|>\e^{-1} |z|} \frac{2(L+1)|z|}{|y-x |^2} |b(x+z)-b(y)||f(y)|\,dy\\
   \ls |z| \int_{|x-y|>\e^{-1} |z|} \frac{|f(y)|}{|y-x |^2} \,dy.
\end{align*}
From this and H\"older's inequality, we have
\begin{align}\label{eq:thm2.36}
\int_{\R}|{\rm L_2}|^p\,dx %&\ls&|z|^p\int_{\R}\lf[\int_{|x-y|>\e^{-1} |z|}\frac{|f(y)|}
&\ls|z|^p\int_{\R}\lf[ \int_{|x-y|>\e^{-1} |z|} \frac{1}{|y-x |^{2/p'}} \frac{|f(y)|}{|y-x |^{2/p}}   \,dy\r]^p\,dx\nonumber\\
&=|z|^p\int_{\R}\lf\{\int_{|x-y|>\e^{-1} |z|}\frac{1}{|x-y|^2}\,dy\r\}^{p/p'}
\int_{|x-y|>\e^{-1} |z|}\frac{|f(y)|^p}{|x-y|^2}\,dy\,\,dx \nonumber\\
&\ls|z|^p\int_{\R}\lf(\e|z|^{-1} \r)^{p/p'}\int_{|x-y|>\e^{-1} z}\frac{|f(y)|^p}{|x-y|^2}\,dy\,\,dx \noz \\
&\ls|z|^p \lf(\e|z|^{-1} \r)^{p/p'}\int_{\R} \e|z|^{-1} |f(y)|^p \,dy \noz\\
&=|z|^p \lf(\e|z|^{-1} \r)^{p}\|f\|^p_{L^p(\R)} \noz\\
&=\e^p\|f\|^p_{L^p(\R)}.
\end{align}

Turning to ${\rm L}_3$, by~\eqref{eq:thm1.1}, the fact that $b\in C_c^{\infty}(\R)$ and the Mean Value Theorem, we conclude that
\begin{eqnarray*}
  |{\rm L}_3| %&\leq& \int_{|x-y|\le\e^{-1}
   \ls \int_{|x-y|\le\e^{-1} |z|} |f(y)|\,dy. \\
\end{eqnarray*}
Then using H\"{o}lder's inequality as for~$\rm{L}_2$ we see that
\begin{align}\label{eq:thm2.37}
\int_{\R}|{\rm L_3}|^p\,dx &\ls  \int_{\R} \lf[\int_{|x-y|\le\e^{-1} |z|} |f(y)|\,dy \r]^p \,dx \noz\\
&\ls\int_{\R}\lf\{\lf[\int_{|x-y|\le\e^{-1} |z|}\,dy\r]^{p/p'}\int_{|x-y|\le\e^{-1} |z|}|f(y)|^p\,dy\r\}\,dx\noz\\
&\ls(\e^{-1} |z|)^{p}\|f\|^p_\lpz \noz\\
&< \e^p\|f\|^p_{L^p(\R)}
\end{align}
by our choice of~$z$.
Similarly, we obtain the same estimate for $\rm{L}_4$:
\begin{equation}\label{eq:thm2.38}
  \int_{\R}|{\rm L_4}|^p\,dx \ls \e^{p}\|f\|^p_\lpz.
\end{equation}

Lastly, we consider $\rm{L}_1$:
\begin{align*}
  |\rm L_1| %&=& \lf|\int_{|x-y|>\e^{-1} |z|}\Clm(x, y)[b(x)-b(x+z)]f(y)\,dy\r| \\
   %&\leq& |b(x)-b(x+z)| \lf|\int_{|x-y|>\e^{-1} |z|}\Clm(x, y)f(y)\,dy\r|\\
   &\leq |b(x)-b(x+z)|\sup_{t>0}\lf|\int_{|x-y|>t}\Clm(x, y)f(y)\,dy\r| \\
   &=: |b(x)-b(x+z)| C_{\Gamma}^{\ast} f(x).
\end{align*}
Here we will use the following standard result.
%; see for example~\cite{Duo01}.
\begin{thm}(\textbf{\cite[Theorem 5.14, p.102]{Duo01}})
If~$T$ is a Calder\'{on}--Zygmund operator, then~$T^*$ is weak (1,1) and strong (p,p) for all~$p \in (1,\infty)$. $T^*$ is defined by
\[T^*f(x) := \sup_{t>0} \lf|\int_{|x-y|>t}K(x,y) f(y) \,dy\r|.\]
\end{thm}
Thus we see that $C_{\Gamma}^{\ast}$ is bounded on $\lpz$ for all $p\in(1, \fz)$. Recall that~$|b(x) - b(x+z)| < \e$ by our choice of~$z$.
Hence
\begin{align}\label{eq:thm2.39}
  \int_{\R}|{\rm L_1}|^p\,dx &\leq \int_{\R}  |b(x)-b(x+z)|^p |C_{\Gamma}^{\ast} f(x)|^p \,dx \noz\\
   &< \e^p \int_{\R}  |C_{\Gamma}^{\ast} f(x)|^p \,dx \noz\\
  % &=& \e^p \|C_{\Gamma\ast} f(x)\|_{\lpz}^p \noz\\
   &\ls  \e^p \| f\|_{\lpz}^p.
\end{align}
Combining the estimates~\eqref{eq:thm2.36}--\eqref{eq:thm2.39} of ${\rm L}_i,\,i\in\{1, 2,3,4\}$, we conclude that
\begin{align*}
\lf[\int_{\R}\lf|[b, \Clm]f(x)-[b, \Clm]f(x+z)\r|^p\,dx\r]^{1/p}
&\ls\sum_{i=1}^4\lf(\int_{\R} |{\rm L}_i|^p \,dx\r)^{1/p} \\
&\ls\e\|f\|_\lpz.
\end{align*}
This shows that $[b,\Clm]E$ satisfies hypothesis (c) in Theorem \ref{t-fre kol}. Hence,
$[b, \Clm]$ is a compact operator. This completes the proof of Theorem \ref{thm2}. \hfill\(\Box\)

\subsection*{Acknowledgements}
 T.T.T. Nguyen is supported by an
Australian Government Endeavour Postgraduate Scholarship. J. Li
and L.A. Ward are supported by the Australian Research Council,
Grant ARC-DP160100153. B.D. Wick is supported in part by
National Science Foundation grant DMS \#1560955.


\begin{thebibliography}{LOPTT09}
\bibitem[BDMT15]{BDMT15}
{\sc  A. B\'{e}nyi, G. Damian, K. Moen and R.H. Torres},
  Compactness properties of commutators of bilinear fractional integrals, {\it Math. Z. } {\bf280} (2015), no.~1, 1432--1823.

\bibitem[BT13]{BT13}
{\sc  A. B\'{e}nyi and R.H. Torres},
  Compact bilinear operators and commutators, {\it Proc. Am. Math. Soc. } {\bf141} (2002), no.~10, 3609--3621.

\bibitem[Blo85]{Blo85} {\sc S. Bloom},  A commutator theorem and weighted BMO, {\it Trans. Amer. Math. Soc. } {\bf292} (1985), no. 1, 103--122.

\bibitem[Bou02]{Bou02} {\sc  G. Bourdaud},
  Remarques sur certains sous-espaces de $\text{BMO}(\mathbb{R}^n)$ et de $\text{bmo}(\mathbb{R}^n)$, {\it Ann. Inst. Fourier (Grenoble)} {\bf 52} (2002), no.~4, 1187--1218.

\bibitem[Bre10]{Bre10} {\sc H. Brezis},
\emph{Functional analysis, Sobolev spaces, and partial differential equations}, Springer, New York, 2010.

\bibitem[Cha16]{Cha16}
    {\sc  L. Chaffee}, %
   Characterizations of bounded mean oscillation through commutators of bilinear singular integral operators, {\it Proc. Roy. Soc. Edinburgh Sect. A}  \textbf{146} (2016), no.~6, 1159--1166.


\bibitem[CT15]{CT15}
    {\sc  L. Chaffee and R.H. Torres}, %
    Characterization of compactness of the commutators of
bilinear fractional integral operators, {\it Potential Anal.}  \textbf{43} (2015), no.~3, 481--494.

\bibitem[Chr90b]{Chr55}
    {\sc M. Christ}, %
    \emph{Lectures on singular integral operators}, %
    CBMS Reg. Conf. Ser.  Math. \textbf{77}, Amer. Math. Soc., Providence, R.I., 1955.

\bibitem[CMM82]{CMM82}
    {\sc R.R. Coifman, A. McIntosh and Y. Meyer}, %
 L'integral de Cauchy d\'{e}finit un operateur born\'{e} sur $L^2$ pour les courbes lipschitziennes, {\it Ann. of Math.} \textbf{115} (1982), no. 2, 361--387.

\bibitem[CRW76]{CRW76}
    {\sc  R.R. Coifman, R. Rochberg and G. Weiss}, %
    Factorization theorems for Hardy spaces in several variables, {\it Ann. of Math.}  \textbf{103} (1976), no. 3, 611--635.

\bibitem[CW77]{CW77}
    {\sc  R.R. Coifman  and G. Weiss}, %
    Extensions of Hardy spaces and their use in analysis, {\it Bull. Amer. Math. Soc.}  \textbf{83} (1977), 569--645.

\bibitem[CSS12]{CSS12}
    {\sc  M. Cwikel, Y. Sagher and P. Shvartsman}, %
    A new look at the John--Nirenberg and John--Str\"{o}mberg theorems for~$\bmo$, {\it J. Funct. Anal.} \textbf{263} (2012), no. 1, 129--166.

\bibitem[Daf02]{Daf02}
    {\sc  G. Dafni}, %
    Local~$\vmo$ and weak convergence in $h^1$, {\it Canad. Math. Bull.}  \textbf{45} (2002), no. 1, 46--59.

\bibitem[Duo01]{Duo01}
    {\sc J. Duoandikoetxea}, %
    \emph{Fourier Analysis}, %
   Grad. Stud. Math. \textbf{29}, Amer. Math. Soc., Providence, RI, 2001.

\bibitem[FL02]{FL}
{\sc S. Ferguson and M. Lacey},
A characterization of product BMO by commutators,
{\it Acta Math.} {\bf189} (2002), no. 2, 143--160.

\bibitem[Gra04]{Gra04}
    {\sc L. Grafakos}, %
    \emph{Classical and Modern Fourier Analysis}, %
   Pearson/Prentice Hall, 2004.

\bibitem[HLW17]{HLW17}
{\sc I. Holmes, M. Lacey and B.D. Wick},  Commutators in the two-weight setting, {\it Math. Ann.} {\bf367} (2017), no.~1--2, 51--80.

\bibitem[Jou83]{Jou83} %
    {\sc J.-L. Journ\'{e}}, %
    \emph{Cald\'{e}ron-Zygmund operators, pseudo-differential operators and the Cauchy integral of Calder\'on},
    Lecture Notes in Math. \textbf{994},   Springer, Berlin, 1983.

\bibitem[KL2]{KL2} {\sc S. G. Krantz and S.-Y. Li}, Boundedness and compactness of integral operators on spaces of homogeneous type and applications, II, {\it J. Math. Anal. Appl.}, {\bf258} (2001), 642--657.


\bibitem[Ler11]{Ler11}
    {\sc A.K. Lerner}, %
    \emph{A ``local mean oscillation'' decomposition and some of its applications}, %
    in: Function spaces, Approximation, Inequalities and Lineability, Lectures of the Spring School in Analysis (Paseky, 2011), Matfyzpress, Praha, 2011, 71--106.

\bibitem[LOPTT09]{LOPTT09}
    {\sc A.K. Lerner, S. Ombrosi, C. P{\'e}rez, R.H. Torres and R. Trujillo-Gonz{\'a}lez}, %
    New maximal functions and multiple weights for the multilinear Calder\'on-Zygmund theory, %
    {\it Adv. Math.} \textbf{220} (2009), no.~4, 1222--1264.


\bibitem[LOR]{LOR} {\sc A.K. Lerner, S. Ombrosi and I.P. Rivera-R\'ios}, Commutators of singular integrals revisited, Bull. London Math. Soc. \textbf{51} (2019), 107 --119.


\bibitem[LW17]{LW17}
    {\sc J. Li and B.D. Wick}, %
    Weak factorization of the Hardy space $H^1(\R^n)$ in terms of multilinear Riesz transforms, %
    {\it Canad. Math. Bull.} \textbf{60} (2017), no. 3, 517--585.

\bibitem[PT03]{PT03}
    {\sc C. P{\'e}rez and R.H. Torres}, %
   \emph{Sharp maximal function estimates for multilinear singular
   integrals}, in: Contemp. Math. \textbf{320}, Amer. Math. Soc., 2003, 323--331.


\bibitem[Tan08]{Tan08}
    {\sc L. Tang}, %
    Weighted estimates for vector-valued commutators of multilinear
   operators, %
    {\it Proc. Roy. Soc. Edinburgh Sect. A} \textbf{138} (2008), no. 4, 897--922.

\bibitem[TYY]{TYY} {\sc J. Tao, Da. Yang and Do. Yang}, Boundedness and compactness characterizations of Cauchy integral commutators on Morrey spaces,  arXiv:1801.04997.


\bibitem[Uch78]{Uch78} {\sc A. Uchiyama},
On the compactness of operators of Hankel type,
{\it T\^{o}hoku Math. J.} \textbf{30} (1978), 163--171.

\bibitem[Yos80]{Yos80} {\sc K. Yosida},
\emph{Functional analysis}, $6^{\textup{th}}$ ed., Grundlehren Math. Wiss. \textbf{123}, Springer, Berlin, 1980.
\end{thebibliography}
\end{document}